\documentclass[smallextended,numbook,envcountsame,envcountreset]{svjour3}
\smartqed 
\usepackage{graphicx}
\def\Label#1{}
\usepackage[usenames,dvipsnames]{color}
\let\proof\relax

\usepackage{amsthm}
\usepackage{times}
\usepackage{amsmath}
\usepackage{amsfonts}
\usepackage{amssymb}
\usepackage{comment}

\usepackage{bbm}
\usepackage{mhequ}
\usepackage{sub_JP}
\usepackage{comment}
\def\norm#1{{\left\vert\kern-0.25ex\left\vert #1 
    \right\vert\kern-0.25ex\right\vert}}
\def\nnorm#1{{\left\vert\kern-0.25ex\left\vert\kern-0.25ex\left\vert #1 
    \right\vert\kern-0.25ex\right\vert\kern-0.25ex\right\vert}}

\def\FF{\mathcal F}
\def\OO{\mathcal O}
\def\LL{\mathcal L}
\def\XX#1{{\mathcal X}(#1)}

\newenvironment{myenum}
{\begin{enumerate}
\setlength{\itemsep}{1pt}
\setlength{\parskip}{0pt}
\setlength{\parsep}{0pt}}
{\end{enumerate}}
\newcommand{\ie}{{\it i.e.}, }
\newcommand{\eg}{{\it e.g.}, }
\let\epsilon=\varepsilon
\def\d{{\rm d}}
\def\eref#1{Eq.(\ref{#1})}
\def\fref#1{Fig.(\ref{#1})}
\def\tref#1{Theorem~\ref{#1}}
\def\lref#1{Lemma~\ref{#1}}
\def\sref#1{Sec.\ref{#1}}
\def\pref#1{Prop.\ref{#1}}

\def\i{{\bf i}}
\def\braket#1#2{|\v#1_{#2} \rangle \langle \n#1_{#2} |}
\let\theta=\vartheta

\def\K{2}
\def\TT{T}
\let\phi=\varphi
\let\kappa=\varkappa
\def\expit{e^{\I (t\phi+\theta)}}
\def\expmit{e^{-\I (t\phi+\theta)}}
\def\expitt{e^{\I (t\phi(t)+\theta(t))}}

\def\star{{}}
\def\psphit{p^*\bigl(\phi(t)\bigr)}
\def\pphizero{p^*(\phi_0)}
\def\pphi{p^*(\phi)}

\def\norm#1{\|#1\|}
\def\psphit{p_{\phi(t)}}
\def\pphizero{{p_{\phi_0}}}
\def\pphi{{p_\phi}}

\def\psphit{p\bigl(\phi(t)\bigr)}
\def\pphizero{p(\phi_0)}
\def\pphi{p(\phi)}
\def\pphizero{p(\phi_0)}

\def\cir{\kern-0.3em\circ\kern-0.2em}
\def\dpphi{\partial_\phi\pphi}
\def\dpphidotphi{\partial_\phi \bigl(\pphi\bigr)\dot\phi}

\def\v#1{v^{(#1)}}
\def\n#1{n^{(#1)}}
\def\pmat#1{\begin{pmatrix}#1\end{pmatrix}}
\def\complex{\mathbb C}
\def\real{\mathbb R}
\def\P{\mathbb P}
\def\Q{\mathbb Q}
\def\Px{{\mathbb P}^\xi}
\def\Py{{\mathbb P}^\eta}

\usepackage{bbold}
\def\one{\mathbb{1}}

\def\I{{\rm I}}

\def\T{\top}

\def\tdotphitheta{(t\dot\phi+\dot\theta) }

\def\pphizerozero{\pmat{\pphizero\\0}}

\def\mid{\middle |}
\def\freq{\Omega}
\def\I{{\bf \i}}

\def\penull{{\phi_0,\epsilon}}
\makeatletter
\renewcommand*\env@matrix[1][*\c@MaxMatrixCols c]{%
  \hskip -\arraycolsep
  \let\@ifnextchar\new@ifnextchar
  \array{#1}}
\makeatother

\begin{document}
\title{Decay of Hamiltonian Breathers under Dissipation}
\author{Jean-Pierre Eckmann \and C. Eugene Wayne}
\institute{J.-P. Eckmann\at D\'epartement de Physique Th\'eorique and
Section de Math\'ematiques\\ Universit\'e de
Gen\`eve\\1211 Geneva 4 Switzerland \and C.E. Wayne\at Department of
Mathematics and Statistics\\ Boston University\\Boston MA 02215 USA}

\dedication{Dedicated to the memory of our friend and colleague, Walter Craig.}

%e-mail: collet@cpht.polytechnique.fr}
\date{Received: date / Accepted: date}
\maketitle

\begin{abstract}
  We study metastable behavior in a discrete nonlinear Schr\"odinger equation
from the viewpoint of Hamiltonian systems theory.   When there are
  $n<\infty$ sites in this equation, we consider initial conditions in which almost
  all the energy is concentrated in one end of the system.  We are interested in 
  understanding how energy flows through the system, so we  add a
  dissipation of size $\gamma $ at the opposite end of the chain,  and we show
  that the energy decreases  extremely slowly.  Furthermore, the motion is localized in 
  the phase space near a family of breather solutions for the undamped system.  We give
  rigorous, asymptotic estimates for the rate of evolution along the family of breathers and the
  width of the neighborhood within which the trajectory is confined.
  \end{abstract}
\tableofcontents
\section{Introduction}

In the present work we look at the problem of a finite, discrete
nonlinear Schr\"odinger equation, with dissipation, which we
considered first in \cite{EckmannWayne2018}. We need to repeat several
equations from that paper, but the aim is now to give a complete proof
of the observations and  assertions in that paper.
One starts with
\begin{equ}\label{eq:dNLS}
-\I  \frac{\partial u_j}{\partial \tau} = - (\Delta u)_j + |u_j|^2 u_j\ , \
j=1,2, \dots , n\ ,
\end{equ}
where we will add dissipation later.
Here $(\Delta u)_j= u_{j-1}-2 u_j +u_{j+1}$, with obvious
modifications for $j=1$ or $n$, \ie $(\Delta u)_1= -u_1+u_2$ and
$(\Delta u)_n=-u_n+u_{n-1}$.
{
\begin{comment}
Several of the ideas were
described in our earlier paper \cite{EckmannWayne2018}, and the current paper
gives complete details and proofs for arbitrary $n$. 
\end{comment}
For the
convenience of the reader, the following introduction repeats the
setup from \cite{EckmannWayne2018}. 
}

We will choose initial conditions for this system in which essentially all of
the
energy is in mode $u_1$, and will add a weak dissipative term to the last
mode as in \cite{EckmannWayne2018,cuneo:2017} by adding to \eref{eq:dNLS} a term of the form
\begin{equ}
\I \gamma \delta_{n,j} u_j~,
\end{equ}
\ie we add dissipation to position $n$, at the {opposite }end from the
energetic mode.

Eventually, this will lead to the energy of the system tending to zero, but we are interested in
what happens on intermediate time scales, and in particular, how the energy is transported from one end 
of the lattice to the other.
\begin{comment}
Eventually, this will lead to the energy of the system tending to zero, but we
are interested in what
happens on intermediate time scales.
Our main result is illustrated in \fref{fig:1}.
and is formulated in \tref{thm:evolution}.
\end{comment}

{
If our initial conditions are chosen so that $u_1(0) = \sqrt{ \freq } $, and all other
$u_j(0) = 0$, then we expect
that at least initially, the coupling terms between the various modes will play
only a small
role in the evolution and the system will be largely dominated by the equation
for $u_1$:
\begin{equ}
-\I  \frac{\d u_1}{\d\tau} = |u_1|^2 u_1\ ,
\end{equ}
with solution $u_1(\tau ) = \sqrt{\freq} e^{\I  \freq  \tau}$---\ie we have a very
fast rotation with large amplitude.  With this
in mind, we introduce a rescaled dependent variable and rewrite the equation in
a rotating coordinate
frame by setting:
}
\begin{equ}\label{eq:wide}
u_j(\tau) = \sqrt{\freq}  e^{\I  \freq  \tau} \widetilde{w}_j(\tau)\ .
\end{equ}
Then $\widetilde{w}_j$ satisfies
\begin{equ}
\freq  \widetilde{w}_j  -\I  \frac{\partial \widetilde{w}_j  }{\partial \tau}
 = - (\Delta \widetilde{w})_j + \freq  |\widetilde{w}_j |^2 \widetilde{w}_j\ .
\end{equ}
We now add dissipation by adding a term which acts on the last
variable, with $\gamma\ge0$,
\begin{equ}
\freq  \widetilde{w}_j  -\I  \frac{\partial \widetilde{w}_j  }{\partial \tau}
 = - (\Delta \widetilde{w})_j + \freq  |\widetilde{w}_j |^2 \widetilde{w}_j +\I \gamma
\delta_{n,j} \widetilde{w}_j\ .
\end{equ}
Rearranging, and dividing by $\freq $ gives
\begin{equ}
- \I \frac{1}{\freq } \frac{\partial {\widetilde{w}}_j }{\partial \tau}
= - \frac{1}{\freq } (\Delta \widetilde{w})_j - \widetilde{w}_j + |\widetilde{w}_j |^2
\widetilde{w}_j
+ \I \frac{\gamma}{\freq }
\delta_{n,j} \widetilde{w}_j
\  .
\end{equ}
Finally, we define $\epsilon = \freq ^{-1}$, and rescale time so that $\tau =
\epsilon t$.  Setting
$w(t) = \widetilde{w}(\tau)$,  we arrive finally
at
\begin{equ}\label{eq:main}
-\I  \frac{\partial w_j}{\partial t} = -\epsilon (\Delta w)_j - w_j + |w_j|^2
w_j+ \I \gamma \epsilon\delta_{n,j} w_j~.
\end{equ}

Now that we have defined the main equation \eref{eq:main} we describe
the picture that will be proved later:

When $\gamma =0$ and $|\epsilon| \ne0$ is small, {this system
 possesses a family of breathers,  namely solutions which in this rotating coordinate
 systems are stationary states in which most of the energy
is localized in site 1. \footnote{The existence and properties of
breather solutions in infinite lattices of oscillators are discussed
in \cite{Mackay:1994} or \cite{Flach:1998}.   The proofs in those cases
are easily modified (and actually somewhat simpler) in the case of finitely many
degrees of freedom.}  Such solutions take the form
\begin{equ}
  u_{\epsilon ,j}^{(0)}\sim (-1)^{j+1}\epsilon ^{j-1}~, \quad j=1,\dots,n~.
\end{equ}

In fact, there are many such solutions, with different frequencies, which we will write
as

\begin{equation}
u_{\epsilon,j}^{(\phi_0)} = e^{\I \phi_0 t} p(\phi_0)_j\ ,
\end{equation}
for $\phi_0$ near zero, and $p(\phi_0)_j \sim   u_{\epsilon ,j}^{(0)}$.
We will demonstrate the existence of these families of solutions in \tref{th:fixed2}, using the implicit function theorem,
and also give more accurate asymptotic formulas for them. 
As \eref{eq:main} is invariant under complex rotations, we actually
have a circle of fixed points, with a phase we call $\theta $. As there is one such circle for every
small $\phi_0 $ we represent, in \fref{fig:1}, these solutions as a
(green) cylinder, with the direction along the cylinder corresponding to changing $\phi_0$ and 
motions ``around'' the cylinder corresponding to changing $\theta$. }

When the dissipation is nonzero ({\it i.e.} when $\gamma >0$)
these periodic solutions are destroyed, but they give rise to 
a family of time-dependent solutions which ``wind'' along the
cylinder,  the red curves in \fref{fig:1}. { We will prove that one can accurately approximate
solutions of the dissipative equations by ``modulating'' the frequency and phase of the breather, namely we 
prove that the solutions of the dissipative equation can be written as:}
\begin{equ}
  u_j(x,t)= \expitt   u_{\epsilon ,j}^{(0)}+z_j(t) ~, \quad j=1,\dots,n~,
\end{equ}
where
\begin{equa}
  \dot\phi(t) &\sim -2\gamma \epsilon ^{2n-1}~,\\
  t\dot\phi(t)+\dot \theta(t) &\sim 0~,\\
    \norm{z(t)}&\text{ remains bounded by }\OO(\gamma\epsilon ^n)~.
\end{equa}
We prove that the initial values $\phi_0$ and $\theta_0$  can be chosen so that $z(t)$ is normal
to the cylinder of breathers at the point $(\phi_0,\theta_0)$, and that its long term boundedness is due to the (somewhat
surprising) fact that the linearized dynamics about the family of breathers is uniformly (albeit weakly) damping in these
normal directions.  This is the main new technical result of the paper and the proof of this fact takes up
Sections \ref{s:evolution}-\ref{sec:semigroup}.
That breathers can play an important role in the non-equilibrium evolution of systems of coupled oscillators has
also been discussed (non-rigorously) in the physics literature.  For two recent examples see
\cite{ILLOP:2017} and \cite{DCF:2017}.

To formulate our results more precisely,  we need some notation:
Let $\delta (t)\equiv\phi(t)-\phi_0$. Let $s(t)=\int_0^t \d \tau
(\tau \dot\phi(\tau ) +\dot \theta (\tau ))$. Let the initial
condition be $\phi_0$, $\theta _0$ and $z_0$, \emph{with $z_0$ perpendicular
to the tangent space to the
cylinder at $\phi_0$, $\theta_0$}.  {Note that $\delta(0) = s(0) = 0$.
\begin{theorem}\label{thm:main}
For sufficiently small $\epsilon >0$ and $\gamma >0$ the following
holds:
Assume 
$$
\|z(0)\| \le \gamma \epsilon ^n ~,
  $$
  Then, there is a constant (depending only on $n$) such that at time $T={\rm const}\,\epsilon ^{-1}$. 
  $$
\|z(T)\| \le \gamma \epsilon ^n ~,
  $$
  While both $\delta(T)$ and $s(T)$ have modulus less than $1$.  
  For all
  intermediate $t$, one has $\|z(t)\| \le 2 \gamma \epsilon ^n $, so trajectory never moves too far from the cylinder of breathers.
Furthermore, one can find $\phi_1$, $\theta _1$, $z_1$,  with $z_1$ in the subspace perpendicular to the tangent space to the
  cylinder at $\phi_1$, $\theta_1 $, with
  \begin{equ}\label{eq:re}
  e^{\I(T\phi_1+\theta _1)}p(\phi_1) + z_1 =
  e^{\I(T\phi(T)+\theta (T))}p(\phi(T)) +z(T)~,
  \end{equ}
  and
  \begin{equation}
\| z_1 \| \le \gamma \epsilon^n\ .
\end{equation}
 Finally, 
  \begin{equ}\label{eq:shift}
  \phi_1-\phi_0=-2\gamma \epsilon ^{2n-1}T+ h.o.t.~.   
  \end{equ}
\end{theorem}
}

\begin{remark}
  The important consequence of \tref{thm:main} is the observation that
  the bounds propagate, so we can restart the evolution, using initial conditions $(\phi_1,\theta_1, z_1)$ instead
  of $|(\phi_0,\theta_0,z_0)$, and therefore one can move to $\phi_2$,
  $\theta _2$, $z_2$, and so on, with a controlled bounds on $\phi_k$,
  which apply at least as long as $\phi_0-\phi_k\le \gamma \epsilon
  ^n$.
  Also note that the deviation from the cylinder is as shown in
  \fref{fig:1}, namely, the orbit can get away from $\|z_k\|\le \gamma
  \epsilon ^n$, during the times between the stopping times $kT$, $k=1,2,\dots$.
\end{remark}

The remainder of the paper is devoted to the proof of \tref{thm:main}.
After some introductory results, the first important bound is on the
linear semigroup with the very weak dissipation in
\sref{sec:semigroup}.
The generator is called $\LL_{\phi\gamma }$, see \eref{eq:ll} and its
associated bound (in Corollary \ref{cor:semi}).
In \sref{sec:estimate}, we study in detail the projection onto the
  complement of the tangent space to the cylinder at $(\phi,\theta) $.
This allows, in \sref{sec:zeta}, to estimate the contraction (after
time $T$) of the $z$-component, orthogonal to the tangent space. We do
this in two steps, first we evolve $z$ while staying in the basis defined at
$\phi_0,\theta_0 $. Then, in \sref{sec:reorthog}, we re-orthogonalize so
that we obtain \eref{eq:re}. Finally, \sref{sec:iter} gives some more
details about restarting the iterations from $\phi_1$, $\theta _1$,
$z_1$ to $\phi_2,\dots$.

The precise statement will be formulated and proved as
\tref{thm:evolution}.

\begin{figure}[t!]
  \begin{center}
    \includegraphics[width=0.8\textwidth]{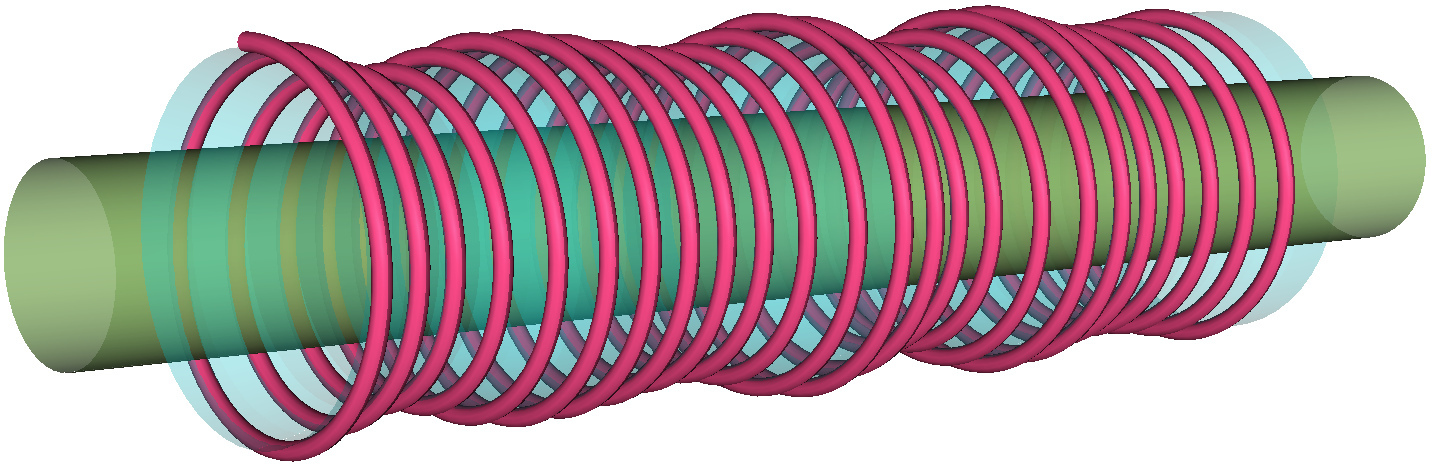}
    \end{center}
\caption{Illustration of the results: Since phase space is
  high-dimensional, we draw the red curve in the same coordinate
  system as the cylinder, but it really stays in a subspace of
  $\complex^n$ which is orthogonal to the 2-dimensional space of the cylinder. When there is no dissipation
  ($\gamma =0$), then the system has a cylinder of fixed points in
  rotating frames (shown in green). This cylinder is
  parameterized by the values of $\epsilon $, \ie the energy of the
  fast coordinate $u_1$. When $\gamma >0$, the fixed points disappear,
  and instead the system hovers near the cylinder, and spiraling
  around it, with a phase speed of $2\gamma \epsilon ^{2n-1}$. We show
  that the orbit of all such solutions stays within a distance
  $\OO(\epsilon ^n)$, as long as $\epsilon $ remains small (it
  actually increases with time).
}
\label{fig:1}       % Give a unique label
\end{figure}

\begin{remark}
  From the results of
    \cite{EckmannWayne2018} and \eref{eq:shift} one can also conclude more details about
    the windings of \fref{fig:1}. The $m^{\rm th}$ turn { finishes after a time
    $t_m \approx \sqrt{\frac{2\pi m}{\gamma \epsilon ^{2n-1}}}$, and the ``horizontal''}
    spacing (in $\epsilon $) between the windings is $2\sqrt{2\pi\gamma
      \epsilon ^{2n-1}}(\sqrt{m+1}-\sqrt{m})$, up to terms of higher order.
\end{remark}

We will study \eref{eq:main} for the remainder of this paper.   We will also
sometimes rewrite
this equation in the equivalent real form by defining $w_j = p_j + \I q_j$,
which yields the system
of equations, for $j=1,\dots,n$:
\begin{equa}[eq:main_real]
  \dot q_j &= -\epsilon (\Delta p)_j
  -p_j+(q_j^2+p_j^2)p_j-\delta_ {j,n}\gamma\epsilon q_n~,\\
  \dot p_j &= \hphantom{-}\epsilon (\Delta q)_j +q_j-(q_j^2+p_j^2)q_j-\delta_
{j,n}\gamma\epsilon p_n~~.
\end{equa}
Note that if  $\gamma=0$, this is a Hamiltonian system with:
\begin{equa}[e:hamiltonian]
 H&=\frac{\epsilon }{2} \sum_{j<n}
\left((p_j-p_{j+1})^2+(q_j-q_{j+1})^2\right)\\ &~~~-
\sum_{j=1}^n \left(\frac{1}{2}(p_j^2+q_j^2)-\frac{1}{4}
(p_j^2+q_j^2)^2\right)~.
\end{equa}

Finding a periodic solution of the form \eref{eq:wide} (\ie a fixed point in
the rotating coordinate system) is reduced
to finding roots of the system of equations
\begin{equa}[eq:noomega]
  0&= -\epsilon (\Delta p)_j -p_j+(q_j^2+p_j^2)p_j~,\\
  0 &= \epsilon (\Delta q)_j +q_j-(q_j^2+p_j^2)q_j~.
\end{equa}
Since we are also interested in solutions which rotate (slowly),
replacing $w$ by $e^{\I \phi_0 t}w$,
we study instead of \eref{eq:main_real} (resp.~\eref{eq:noomega}) the related equation
\begin{equa}[eq:main_real_phi]
  \dot q_j &= -\epsilon (\Delta p)_j
  -(1+\phi_0)p_j+(q_j^2+p_j^2)p_j-\delta_ {j,n}\gamma\epsilon q_n~,\\
  \dot p_j &= \hphantom{-}\epsilon (\Delta q)_j +(1+\phi_0)q_j-(q_j^2+p_j^2)q_j-\delta_
{j,n}\gamma\epsilon p_n~~,
\end{equa}where the $\phi_0$ dependence comes from differentiating the
exponential factor $e^{\I \phi_0 t}$.
\begin{remark}
  We use $\phi_0$ to {designate} a constant rotation speed, while later, $\phi$ will
  stand for a time-dependent rotation speed.
\end{remark}

\begin{remark}
  The reader who is familiar with the paper \cite{EckmannWayne2018} can
  jump to \sref{s:evolution}, since much of the material in this
  section and the next is basically repeated from that reference.
\end{remark}

\begin{theorem}\label{th:fixed2} Suppose that the damping coefficient
  $\gamma$ equals $0$ in
\eref{eq:main_real_phi}.  There exist constants $\epsilon_* > 0$, $\phi_* >0$, such that
for
$|\epsilon| < \epsilon_*$ and $ |\phi| < \phi_*$, \eref{eq:main_real_phi} has a
periodic solution
of the form $w_j(t;\phi_0) = e^{\I  t\phi_0 } p_j^\star (\phi_0)$, with $p_1^\star
(\phi_0) = 1+{\OO}(\epsilon, \phi_0)$,
and $p_j^\star (\phi_0) = {\OO}(\epsilon^{j-1})$ for $j = 2, \dots, n$.
\end{theorem}

\begin{proof}

If we insert $w_j(t;\phi_0) = e^{\I  t\phi_0 } p_j^\star (\phi_0)$ into
\eref{eq:main}, and take real
and imaginary parts, we find that the amplitudes
 $p^\star\in\real^n $ of these periodic orbits are (for $\gamma=0$) solutions of
\begin{equ}\label{eq:fp}
F_j(p;\phi_0,\epsilon) = -\epsilon (\Delta p)_j - (1+\phi_0) p_j + p_j^3 = 0\ ,\ \
j = 1, \dots , n\ .
\end{equ}
Setting $p^0_j = \delta_{j,1}$, we have
\begin{equ}
F_j(p^0;0,0) = 0\ ,
\end{equ}
for all $j$.  Furthermore, the Jacobian matrix at this point is the diagonal
matrix
\begin{equ}
\left( D_{p} F (p^0;0,0) \right)_{i,j} = (3 \delta_{i,1} -1) \delta_{i,j}~,
\end{equ}
which is obviously invertible.

Thus, by the implicit function theorem, for $(\phi_0,\epsilon)$ in some
neighborhood of the
origin, \eref{eq:fp} has a unique fixed point $p = p^\star (\phi_0,\epsilon)$,
and since $F$ depends
analytically on $(\phi_0,\epsilon)$, so does $p^\star (\phi_0,\epsilon)$.

It is easy to compute the first few terms of this fixed point:
 \begin{equa}[eq:sol]
 {p_1^\star }    &=1+{\frac {1}{2}}(
\phi_0-\epsilon)+\OO_2~,\\
{p_2^\star }   &=-\epsilon+\OO_2~,\\
{p_3^\star }   &=\epsilon^{2}+\OO_3~,\\
&\dots\\
{p_j^\star }   &=(-1)^{j-1}\epsilon^{j}+\OO_{j+1}~,
\end{equa} 
 where $\OO_k$ denotes terms of order $k$ in
 $\phi_0,\epsilon $ together.

\end{proof}
\begin{remark} Since \eref{eq:main}  is invariant under complex rotations $w_ j
\to e^{\I  \theta_0} w_j$,
we actually have a circle of fixed points (when $\gamma=0$).  However,
these are the only fixed points with $|w_1| \approx 1$. We will
continue with $\theta_0=0$, and reintroduce $\theta_0\ne0$ only in
\sref{s:evolution}.
\end{remark}

\section{The {eigenspace} of the eigenvalue 0}

Consider the linearization of the system \eref{eq:main_real_phi} around
the periodic orbit (fixed point), we found in \tref{th:fixed2}.
Denote by $Z_*$ this solution,
\begin{equ}
Z_* =(p_1^\star ,p_2^\star,\dots ,p_n^\star ,q_1^\star ,q_2^\star,\dots,
q_n^\star )^\T~,
\end{equ}
where $q_j=0$ and $p_j=p_j(\phi_0,\epsilon )$ as found in \tref{th:fixed2}.
{ In order to avoid overburdening the notation, we will write out the formulas
which follow for the case $n=3$ --- the expressions for general  (finite) values of $n$ are very similar.}
 We also omit the $\epsilon $ dependence from
$p(\phi_0,\epsilon )$.
The linearization of the evolution \eref{eq:main_real_phi} at $Z_*$ leads
(for $\gamma=0$) to 
an equation of the form
$$
\frac{\d x}{\d t} = M_{\phi_0,\epsilon } x =  \begin{pmatrix}
  0&A_\penull\\B_\penull&0
  \end{pmatrix} x\ ,
  $$
and with $1_{\phi_0}\equiv(1+\phi_0)$:
\begin{equa}[eq:aphi]
  A_{\phi_0,\epsilon}=
  \begin{pmatrix}
    1_{\phi_0}-\epsilon -(p^\star _1)^2 &\epsilon &0\\
    \epsilon&  1_{\phi_0}-2\epsilon-(p^\star _2)^2 &\epsilon \\
    0&\epsilon&1_{\phi_0}-\epsilon -(p^\star _3)^2
  \end{pmatrix}~,
\end{equa}
where the $p^\star _j=p(\phi_0)_j$ are the stationary solutions of \eref{eq:main_real_phi}.
Similarly,
\begin{equa}[eq:bphi]
B_{\phi_0,\epsilon }=
  \begin{pmatrix}
    -1_{\phi_0}+\epsilon +3(p_1^\star )^2 &-\epsilon &0\\
    -\epsilon&  -1_{\phi_0}+2\epsilon +3(p_2^\star )^2 &-\epsilon \\
    0&-\epsilon&-1_{\phi_0}+\epsilon +3(p_3^\star )^2
  \end{pmatrix}~.
\end{equa}
Similar expressions hold for other values of $n$.

Among the key facts that we will establish below, is that $M_{\phi_0,\epsilon}$ has
a two-dimensional zero eigenspace, with an explicitly computable basis, for all values of $\epsilon$.
Then, in subsequent sections we will show that the remainder of the spectrum lies on the imaginary
axis and that all non-zero eigenvalues are simple and separated from the remainder of the spectrum
of $M_{\phi_0,\epsilon}$ by a distance at least $C \epsilon$.  All of these facts turn out to be essential
for our subsequent calculations and establishing them is complicated by the extreme degeneracy
of the eigenvalues of $M_{\phi_0,0}$ about which we wish to perturb.

The following lemma will allow to simplify notation:
\begin{lemma}\label{lem:binverse}
  One has the identity
  \begin{equ}\Label{eq:ddphi}
    \partial_\phi p(\phi_0)=B_{\phi_0,\epsilon}^{-1} p(\phi_0)~.
  \end{equ}
\end{lemma}
\begin{proof}
  This follows by differentiating \eref{eq:fp} and comparing to the
  definition of $B_{\phi_0,\epsilon}$ in \eref{eq:bphi}.
\end{proof}

\begin{lemma} Define  $B_{\phi_0,\epsilon }=L-\phi_0 \one+3(\pphi_0)^2$, with $L=\epsilon \Delta
-\one$. That is, we view $B_{\phi_0,\epsilon}$ as a real
$n\times n$ matrix and $(\pphizero)^2$ is the
diagonal matrix with components $((\pphizero_1)^2,\dots,(\pphizero_n)^2)$.  
Then the  zero eigenspace of the matrix $M_{\phi_0,\epsilon }$ is
spanned by
the $2n$-component vectors
  \begin{equa}[eq:v]
\v1_{\phi_0}&=\pmat{0\\\pphizero}~,\\
\v2_{\phi_0}&=\pmat{B_{\phi_0,\epsilon }^{-1}\,\pphizero\\0}~,
\end{equa}
\end{lemma}

\begin{proof}
To see that $M_{\phi_0,\epsilon }\v1_{\phi_0}=0$, note that \eref{eq:dNLS}
is invariant under $u\to e^{\I \theta } u$. Thus, viewed in
$\complex^n$, the quantity $e^{\I \theta }(p(\phi_0,\epsilon )+\I 0)$ is
a solution for all $\theta $. Taking the derivative w.r.t.~$\theta $,
at $\theta =0$ and considering the real and imaginary parts of
the resulting equation shows that $\v1_{\phi_0}$ is a solution of
{$M_{\phi_0,\epsilon }\v1_{\phi_0}=0$. }  From the form of $M_{\phi_0,\epsilon }$ and the
invertibility of $B_{\phi_0,\epsilon } $ we see immediately that $\v2_{\phi_0}$ is mapped
onto the direction of $\v1_{\phi_0}$.
\end{proof}

We will also need the adjoint eigenvectors of $M$:
\begin{lemma}
  The adjoint eigenvectors are given by
  \begin{equa} [eq:n1n2]
\n1_{\phi_0} &= (2+\OO(\phi_0,\epsilon )) \cdot(0,B^{-1}_{\phi_0} p(\phi_0))^\T\ , \\ 
\n2_{\phi_0} &=  (2+\OO(\phi_0,\epsilon ))\cdot (p(\phi_0),0)^\T ~.
  \end{equa}
  They are normalized to satisfy
  \begin{equa}[eq:n1v1]
    \langle \n1_{\phi_0}|\v1_{\phi_0}\rangle&=
    \langle \n2_{\phi_0}|\v2_{\phi_0}\rangle=1~,\\
    \langle \n2_{\phi_0}|\v1_{\phi_0}\rangle&=
    \langle \n1_{\phi_0}|\v2_{\phi_0}\rangle=0~.\\
  \end{equa}
\end{lemma}
\begin{remark}
  The approximate versions are
  \begin{equa}
    \n1_{\phi_0}&\sim(0,\dots,0,1,0,\dots,0)^\T~,\\
    \n2_{\phi_0}&\sim(2,0,\dots,0)^\T~~.
  \end{equa}
\end{remark}

\begin{proof}
  
Because of the block form of $M$ and the fact that $A$ and $B$ are symmetric, we have
$$
M^*_{\phi_0,\epsilon } = \left( \begin{array}{cc} 0 & B_{\phi_0,\epsilon }
  \\ A_{\phi_0,\epsilon } & 0 \end{array} \right)~.
$$
{But then, since we know from the computation of the eigenvectors of $M$ that
$B_{\phi_0,\epsilon } \pphizero= 0$, we can
check immediately that}
$$
\tilde{n}^{(2)}_{\phi_0}
= (p(\phi_0),0)^\T
$$
satisfies $M^* \tilde{n}^{(2)}_{\phi_0} = 0$.  Likewise, 
$$
\tilde{n}^{(1)}_{\phi_0} =(0,B^{-1}_{\phi_0,\epsilon } p^\star)^\T
$$
satisfies
$$
M^*_{\phi_0,\epsilon } \tilde{n}^{(1)}_{\phi_0} = (p(\phi_0),0)^\T = \tilde n^{(2)}_{\phi_0}\ .
$$
Thus, $\tilde{n}^{(1)}_{\phi_0}$ and $\tilde{n}^{(2)}_{\phi_0}$ span the zero
eigenspace of the adjoint matrix. The normalization is checked from
the definitions.
\end{proof}

\section{Evolution equations for $\gamma>0$}\label{s:evolution}

Consider \eref{eq:main}, with dissipation:
Here, $C_\Gamma$ is not a scalar, but a diagonal matrix, whose diagonal
will be taken as $(0,0,\dots,\gamma\epsilon )\in \complex^n$.
Thus, our evolution equation is
\begin{equa}[eq:1g]
-\I \dot W =LW +|W|^2W + \I C_\Gamma W~,
\end{equa}
with $L=\epsilon \Delta -\one$, as before. We are interested in the
time dependence of two real ``slow'' variables which we call $\phi(t)$ and
$\theta (t)$, and so we set
\begin{equ}\label{eq:1gphi}
W(t)=\expitt \bigl(\psphit+z(t)\bigr)~,\quad W,z\in \complex^n~.
\end{equ}
\begin{remark}
 Recall that the notation $\phi_0$ stands for a constant phase speed,
 while $\phi=\phi(t)$ will always mean a time-dependent quantity.
\end{remark}
{Our decomposition is inspired by modulation theory approaches to study the stability of
solitary waves and patterns with respect to perturbations \cite{weinstein:1985,pego:1992,pego:1994,promislow:2002}.  In particular, we will choose
the initial decomposition of the solution so that the initial value of $z(0)$ lies in the subspace conjugate to the zero-eigenspace
of the linearization.  We then prove (somewhat surprisingly) that all modes orthogonal to the 
zero subspace are uniformly damped which allows us to show that the values of $z(t)$ remain
bounded for very long times.}
Omitting the arguments $(t)$, we find
\begin{equa}
\dot W &=\I(\phi+t\dot \phi+\dot \theta )\expit (\pphi+z)\\
&~~+ \expit \bigl( \dpphi\,\dot\phi+\dot z \bigr)~.
\end{equa}
Then \eref{eq:1g} leads to (using again that powers and products are taken componentwise),
\begin{equa}
& (\phi+t\dot \phi+\dot \theta )\expit (\pphi+z)
-\I \expit \bigl( \dpphi\dot\phi+\dot z \bigr)\\
&= \expit L(\pphi)+\expit Lz \\&+\bigl ( \expit (\pphi+z)\bigr)^2(\expmit
(\pphi+\bar z))  +\I C_\Gamma\expit(\pphi+z)~.
\end{equa}
The {factors of} $\expit$ cancel and we get
\begin{equa}
& (\phi+t\dot \phi+\dot \theta ) (\pphi+z)
-\I  \bigl( \dpphi\,\dot\phi+\dot z \bigr)\\
&=  L(\pphi)+ Lz + (\pphi+z)^2
(\pphi+\bar z)  +\I C_\Gamma(\pphi+z)~.
\end{equa}
We now expand this equation to first order in $z$ and this leads to
\begin{equa}[eq:3]
&(\phi+t\dot \phi+\dot \theta ) (\pphi+z)
-\I  \bigl( \dpphi\,\dot\phi+\dot z \bigr)\\
&=  L(\pphi)+ Lz \\
& +(\pphi)^3+2(\pphi)^2
z+\pphi\bar z\\
&+\I C_\Gamma(\pphi +z)+\OO(|z|^2) ~.
\end{equa}
Set now $z=\xi+\I\eta$.

In what follows, we will switch back and forth between the real and
complex representations of the solutions and will refer to $z= \xi + \I \eta \in \complex^n$ and 
$\zeta = (\xi,\eta) \in \real^{2n}$
interchangeably, allowing the context to distinguish between the two ways of writing the solution.
When we consider $\xi$ and $\eta$, which are $n$ dimensional vectors,
one should note that $\xi=(\xi_1,\dots,\xi_n)^\T$ while
$\eta=(\eta_1,\dots,\eta_n)^\T$. Correspondingly, we will use
the restrictions of $\P_{\phi_0}$ to these two subspaces, which we call $\Px_{\phi_0}$
and $\Py_{\phi_0}$.

Taking the real and imaginary components of \eref{eq:3}, we obtain the
following equations in $\real^n$:
\begin{equa}
&(t\dot \phi+\dot \theta )(\pphi+\xi)+\dot\eta=
(L-\phi)\xi+3(\pphi)^2\xi-C_\Gamma\eta+\OO_2~,\label{eq:xieta}\\
&(t\dot \phi+\dot \theta )\eta-(\dpphi\,\dot\phi-\dot\xi=
(L-\phi)\eta+(\pphi)^2\eta+C_\Gamma(\pphi+\xi)+\OO_2~,
\end{equa}
where $\OO_2$ refers to terms that are at least quadratic in $(\xi,\eta)$.

We next study what happens in the complement of {the two-dimensional zero eigenspace identified 
at the end of the previous section},
when one adds dissipation on the
coordinate $n$. In the standard basis, when $n=3$, the dissipation is
given, as before, by
\begin{equ}
C_\Gamma=  \begin{pmatrix}
  0&0&0\\
  0&0&0\\
    0&0&\gamma \epsilon 
  \end{pmatrix}~.
\end{equ}
In the full space, we have the $2n\times 2n$ matrix
\begin{equ}
  \Gamma=
  \begin{pmatrix}
    C_\Gamma&0\\0&C_\Gamma
  \end{pmatrix}~.
\end{equ}
We fix a $\phi(0)=\phi_0$ and we consider the projection
\begin{equ}
  \P=\P_{\phi_0} =\one -\braket{1}{\phi_0} -\braket{2}{\phi_0}~.
\end{equ}
This is the projection onto the { complement} of the space spanned by the
0 eigenvalue.

\emph{We will require that $\zeta=(\xi,\eta)$ remains in the
  range of $\P_{\phi_0}$.}  As time passes, the base point
$(\phi,\theta)$ will 
  change, and this will lead to secular growth in $\zeta$, an
  issue which we discuss in detail below.

We rearrange \eref{eq:xieta} as
\begin{equa}[eq:xieta2]
  \dot\xi&=
  \Px_{\phi_0} \biggl(\bigl(-(L-\phi) -(\pphi)^2\bigr) \eta \\&-C_\Gamma\xi
+\tdotphitheta
\eta-\dpphidotphi-C_\Gamma\pphi+\OO_2\biggr)~,\\
\dot\eta&=
\Py_{\phi_0}\left(\bigl((L-\phi) +3(\pphi)^2\bigr)\xi -C_\Gamma\eta
-\tdotphitheta (\pphi+\xi)+\OO_2\right )~.\\
\end{equa}
We will compute these projections in detail in \sref{sec:estimate}.

\section{Spectral properties of the linearization at $\gamma=0$}\label{sec:spec}

In this section, we consider the action of the matrix
$M_{\phi_0,\epsilon }$, when projected (with $\P_{\phi_0}$) onto the complement of the
subspace associated with the 0 eigenspace: 
\begin{equ}
  \P_{\phi_0} = \one - | \v1_{\phi_0}\rangle\,\langle \n1_{\phi_0} | - | \v2_{\phi_0}\rangle\,\langle
  \n2_{\phi_0} |~.
\end{equ}
We will see that this projection is very close to the projection onto
the complement of the 1st and $(n+1)$st component of the vectors in
$\real^{2n}$.  

We use perturbation theory, starting from the matrix $
M_{\phi_0=0,\epsilon=0}$.  We write the
formulas for  $n=4$.
The discussion starts with $\epsilon =\phi_0=0$. Then, we have the
quantities
\begin{equ}\label{eq:Ahat}
A_{0,0}=A_{\phi_0=0,\epsilon =0}
\begin{pmatrix}[c|ccc]
  0 &0 &0&0\\
  \hline
    0&  1 &0&0 \\
    0&0&1&0\\
    0&0&0&1
  \end{pmatrix}=
\begin{pmatrix}[c|c]
  0&\\
  \hline
  &\one
\end{pmatrix}
~,
\end{equ}
and
\begin{equ}\label{eq:Bhat}
  B_{0,0}=B_{\phi=0,\epsilon =0}=
  \begin{pmatrix}[c|rrr]
    2 &0 &0&0\\
    \hline
    0&  -1 &0&0 \\
    0&0&-1&0\\
    0&0&0&-1\
  \end{pmatrix}=
  \begin{pmatrix}[c|c]
    2&\\
    \hline
    &-\one\\
  \end{pmatrix}
  ~,
\end{equ}
which follow by substitution. We set
\begin{equ}\Label{eq:Mhat}
M_{0,0}\equiv  \begin{pmatrix}[cc]
  0&A_{0,0}\\ B_{0,0}&0
  \end{pmatrix}~,
\end{equ}
and study first the spectrum of $M_{0,0}$. The spectrum (and the
eigenspaces) of
$M_{\phi,\epsilon }$ will
then be shown to be close to that of $M_{0,0}$.

The eigenvalues  of $M_{0,0}$ are: A double 0, and $n-1$ pairs of eigenvalues
$\pm\I$.
When $n=4$, the corresponding eigenvectors are:
\begin{equa}
  e^{(1)} &= (0,0,0,0,1,0,0,0)^\T~,  \\
  e^{(3),(4)} &= (0,\pm\I,0,0,0,1,0,0)^\T~,\\
  e^{(5),(6)} &= (0,0,\pm\I,0,0,0,1,0)^\T~,\\
  e^{(7),(8)} &= (0,0,0,\pm\I,0,0,0,1)^\T~.
\end{equa}

Note that $e^{(2)}$ is missing, but the vector $e^{(2)}=(1,0,0,0,0,0,0,0)^\T$ is
mapped onto $2e^{(1)}$ and so $e^{(1)}$ and $e^{(2)}$ span the 0 eigenspace.

Since we have a symplectic problem (when
$\gamma=0$), we need to do the calculations in an appropriate basis. This
is inspired by the paper \cite{Mehrmann2008}. The coordinate
transformation is defined by the following matrix:
Let $s=1/\sqrt{2}$, and define (for the case $n=4$),
\begin{equ}\label{eq:x}
X=\begin{pmatrix}[cc|ccc|ccc]
  &1&&&&&&\\
  \hline
  &&s&&&s&&\\
  &&&s&&&s&\\
  &&&&s&&&s\\
    \hline
    1&&&&&&&\\
    \hline
  &&\I s&&&-\I s&&\\
    &&&\I s&&&-\I s&\\
  &&&&\I s&&&-\I s
\end{pmatrix}~.
\end{equ}

The columns are the normalized eigenvectors of our problem, for
$\epsilon =\phi=0$. Empty positions are ``0''s and the second vector is
mapped on the first (up to a factor of 2).
With our choice of $s$ we have $X^* X=1$, where $X^*$ is the Hermitian
conjugate of $X$.

\begin{definition}
If $Y$ is a matrix, we write its transform as $\XX{Y}=X^* Y X$.
\end{definition}

In the new basis, we get:
\begin{equ}\label{eq:xhatm}
 \XX{M_{0,0}} =-\I 
  \begin{pmatrix}[cc|ccc|ccc]
     0&-2\I&&&&&&\\
  0&0&&&&&&\\
    \hline
  &&1&&&&&\\
  &&&1&&&&\\
  &&&&1&&&\\
    \hline
  &&&&&-1&&\\
  &&&&&&-1&\\
  &&&&&&&-1\\
  \end{pmatrix}=
  -\I\begin{pmatrix}[cc|c|c]
       0&-2\I&&\\
  0&0&&\\
  \hline
  &&\one&\\
  &&&-\one
  \end{pmatrix}
  ~.
\end{equ}
Therefore, we have diagonalized $M_{0,0}$ up to its nilpotent block,
and we also see that the other parts of $\XX{ M_{0,0}}$ are imaginary,
which reflects the 
symplectic nature of the model.

We now turn to the case of $M_{\phi_0,\epsilon }$ which we view as a perturbation
of $M_{0,0}$ in the following way:
\begin{equ}
  M_{\phi_0,\epsilon}=M_{\phi_0,0} + E_1 +E_2~,
\end{equ}
where
\begin{equ}
   M_{{\phi_0},0}=\pmat{0& A_{{\phi_0},0}\\ B_{{\phi_0},0}&0}~,
  \end{equ}
  with
  \begin{equa}
    A_{{\phi_0},0} &= \text{diag}( 0,1_{\phi_0},\dots,1_{\phi_0})~,\\
    B_{{\phi_0}.0} &= \text{diag}( -2\cdot1_{\phi_0},-1_{\phi_0},\dots,-1_{\phi_0})~.
  \end{equa}
The matrix $E_1$ collects the terms in $M_{{\phi_0},\epsilon }$ which are
linear in $\epsilon $, while $E_2$ collects all higher order terms.
  The matrix $E_1$ is easily derived from \eref{eq:sol} and
  \eref{eq:Ahat}--(\ref{eq:Bhat})\footnote{The generalization to
    arbitrary $n$ is obtained by ``filling in'' more rows of the form
    $1,-2,1$ resp.~$-1,2,-1$, while retaining the first and last rows.}:
\begin{equ}\Label{eq:edef}
  E_1=  \epsilon \begin{pmatrix}[cccc|cccc]
   &&&&0&1 &&\\
   &&&&1 &-2&1&\\
   &&&&&1 &-2&1\\
   &&&&&&1 &-1\\
   \hline
-2 &-1 &&&&&&\\
-1 &2 &-1 &&&&&\\
&-1 &2 &-1 &&&&\\
&&-1 &1 &&&&
  \end{pmatrix}~.
\end{equ}
We now apply the coordinate transformation to $M_{{\phi_0},0}$ and $E_1$.
We first observe that
\begin{equ}\label{eq:mphi}
  \XX{M_{{\phi_0},0}} = (1+{\phi_0}) \XX{M_{0,0}}~.
\end{equ}
Applying the transformation to $E_1$,
one gets (using again $s=1/\sqrt{2}$):
\begin{equ}\label{eq:xex}
  \XX{ E_1 } = \I \epsilon  
  \begin{pmatrix}[cc|rrr|rrr]
    &2\I&\I s&&&\I s&&\\
    && s&&&- s&&\\
    \hline
    -\I s& s&-{\bf2}&1&&&&\\
    && 1&-2 &1&&&\\
    && &1 &-1&&&\\
    \hline
    -\I s&- s&&&&{\bf2}&-1&\\
        &&&&& -1&2 &-1\\
    && &&&&-1 &1
  \end{pmatrix}~.
\end{equ}

The reader should be aware that the seeming irregularities of the
matrix $\XX{E_1}$ are due to the differences between the expansions
of $p_1$ and the other $p_j$ in powers of $\epsilon $.
\\
We next split $\XX{E_1}=\XX{ E_{11}}+\XX{ E_{12}}$, where\footnote{We
  maintain the somewhat redundant notation $\XX{\cdot}$ so that the
  reader immediately sees on which space the object in question acts.}
\begin{equ}\label{eq:e11}
  \XX{E_{11}}= \I\epsilon  \begin{pmatrix}[cc|c|c]
    0&2\I&&\\
    0&0&&\\
    \hline
    &&S&\\
    \hline
    &&&-S
  \end{pmatrix}~,
\end{equ}
and the $\pm S$ are the tri-diagonal parts of \eref{eq:xex}.
It is important to observe that the $S$ is
tri-diagonal, symmetric, with non-zero off-diagonal elements. Clearly,
$\XX{E_{12}}$ contains only the  two top rows and the first two
columns of $\XX{E_1}$.
It is thus of the form
\begin{equ}\label{eq:e12}
\XX{E_{12}}=\I\epsilon   \begin{pmatrix}[cc|c|c]
  0& 0 &***&***\\
  0& 0 &***&***\\
  \hline
  *&*&&\\
  *&*&&\\
  *&*&&\\
  \hline
  *&*&&\\
  *&*&&\\
  *&*&&
  \end{pmatrix}~,
\end{equ}
where the ``*''  denote elements of at most $\OO_1$, cf.~\eref{eq:xex}.

All in all, we have decomposed
\begin{equ}\label{eq:decompose}
  \XX{M_{{\phi_0},\epsilon }}=(1+{\phi_0})\XX{M_{0,0}}+\XX{E_{11}} +\XX{E_{12}}+\XX{E_2}~.
\end{equ}

The term $\XX{E_2}$ contributes to second order, and it only remains
to understand the role of $\XX{E_{12}}$. 
Note now that $\XX{ E_{12}}$, which is of the form of \eref{eq:e12},
couples the 0 block
to $S$ and $-S$, but only to the first component of these
matrices.\footnote{This is a remnant of the nearest neighbor coupling
  of the model.} The following argument from classical perturbation
theory shows that this can only contribute to second order in $\epsilon
$ to the spectrum.

The first order shift of an eigenvalue close to $\I$ with eigenvector
$v$ is simply $\langle v |\XX{ E_{12}} v\rangle$. But, $v$ is of the form
\begin{equ}
  v= (0,0,v_1,v_2,v_3,0,0,0)^\T~,
\end{equ}
due to the form $\XX{ M_{{\phi_0},0} + E_{11}}$. Thus, $\XX{E_{12}} v$ is
of the form
\begin{equ}
 \XX{ E_{12}} v= (*,*,0,0,0,0,0,0)^\T~,
\end{equ}
where ``$*$'' denotes possibly non-zero elements. From the form of
$v$, this implies that
\begin{equ}\Label{eq:e12a}
\langle v | \XX{E_{12}} v\rangle=0~.  
\end{equ}
This means that $\XX{E_{12}}$ contributes only in second order in $\epsilon $ to
the eigenvalues.

\section{Complement of the zero eigenspace}

Recall that our goal is to write the solution of \eqref{eq:main} 
as $$w_j(t) = e^{\I \phi(t) t + \theta(t)} ( p_j(\phi(t)) + z_j(t))~,$$
and then follow the evolution of ${\phi}$, $\theta $, and $z =(\xi + \I\eta)$.
Since $\zeta(t)= (\xi(t), \eta(t))$ is constructed to lie in the subspace orthogonal
to the tangent space of the cylinder of breathers at $(\phi_0,\theta_0)$,
we construct  the projection $\P$ onto the tangent space at
$\phi_0$, $\theta_0$. We show that, somewhat surprisingly,  with the exception of the
zero eigenspace, all other eigenvalues of the linearized matrix are simple, lie
on the imaginary axis, and are separated by a distance of at least $C\epsilon$ from
the remainder of the spectrum.
It is convenient to work directly with the transformation $\XX{\cdot}$ of
\eref{eq:x}.  

\begin{theorem}\label{thm:imag}
  The operator $\XX{\P M_{{\phi_0},\epsilon }}$ has two eigenvalues which
  are within $\OO(\epsilon )$ of 0, and $(n-1)$ purely imaginary eigenvalues close to
  $\I$, separated by $C\epsilon$, with $C>0$. The corresponding
  eigenvectors (in the $\XX{\cdot}$ basis)
  are orthogonal. Furthermore, these eigenvalues
  have non-vanishing last component (\ie the components $2+(n-1)$ and
  $2+2(n-1)$ in the $\XX{\cdot}$ representation). Analogous statements hold for the
  $n-1$ eigenvalues near $-\I$.
 
\end{theorem}
\begin{proof}
The remainder of this section is devoted to the proof of this theorem.

A calculation  (using our formulas for $\v{j}$, $\n{j}$)
shows that
\begin{equ}\label{eq:decomposep}
  \XX{\P}=\XX{\P_0} +\XX{\P_1}+\XX{\P_2}~,
\end{equ}
where
\begin{equ}
  \XX{\P_0}=\begin{pmatrix}[cc|ccc|ccc]
  0&0&&&&&&\\
  0&0&&&&&&\\
  \hline
  &&1&&&&&\\
  &&&1&&&&\\
  &&&&1&&&\\
    \hline
  &&&&&1&&\\
  &&&&&&1&\\
  &&&&&&&1\\
  \end{pmatrix}~,\qquad \XX{\P_1}=
  \begin{pmatrix}[cc|c|c]
    0&0&&\\
    0&0&&\\
    \hline
  && P_1&\\
\hline
&&& P_2
\end{pmatrix}~,
\end{equ}
where the orders of the elements of $P_j$ are
\begin{equ}\Label{eq:p1}
   P_j =
  \begin{pmatrix}
 \epsilon^2 &\epsilon ^3 &\epsilon ^4\\
  \epsilon^3 &\epsilon ^4 &\epsilon ^5\\
 \epsilon^4 &\epsilon ^5 &\epsilon ^6\\
  \end{pmatrix}~.
\end{equ}
Furthermore the $P_j$ are symmetric.
Note that $ (P_j)_{i,k} =\OO(\epsilon ^{i+k})$.
Finally, showing orders only,
\begin{equ}
  \XX{\P_2}=
  \begin{pmatrix}[cc|ccc|ccc]
    \epsilon +{\phi_0}&0&\epsilon &0&0&\epsilon &0&0\\
    0&\epsilon +{\phi_0}&\epsilon &0&0&\epsilon &0&0\\
    \hline
    \epsilon &\epsilon &&&&&\\
    0 &0 &&&&&\\
    0 &0 &&&&&\\
    \hline
        \epsilon &\epsilon &&&&&\\
    0 &0 &&&&&\\
    0 &0 &&&&&\\
  \end{pmatrix}+\OO_2~.
\end{equ}
The omitted terms are similar to those in $P_j$ (again a symmetric matrix).
Therefore, the eigenvectors are orthogonal. (This actually follows
also from the Hamiltonian nature of the problem, but we need more
information to control the $\gamma $-dependence.)

\begin{remark}
  Clearly, $\XX{\P_0}$ is the projection on the eigenspace spanned by $\pm\I$,
  when $\epsilon =0$. The part $\XX{\P_1}$ contains the couplings within
  the subspace of the eigenvalues $\pm\I$, while $\XX{\P_2}$ describes the
  coupling between the zero-eigenspace of dimension 2 and its complement.
\end{remark}

{}From \sref{sec:spec} we also have the decomposition \eref{eq:decompose}:
\begin{equ}\Label{eq:decomposem}
  \XX{M_{{\phi_0},\epsilon }}=(1+{\phi_0})\XX{M_{0,0}}+\XX{E_{11}} +\XX{E_{12}}+\XX{E_2}~.
\end{equ}
Combining with \eref{eq:decomposep}, we see that
\begin{equ}
  \XX{\P M_{{\phi_0},\epsilon }}=\XX{\P}\XX{M_{{\phi_0},\epsilon }}
 \end{equ}
leads to 12 terms, many of which are of second order in $\epsilon $
and ${\phi_0}$. We start with
the dominant ones.

Since $\XX{\P_0}$ is just the projection onto the  complement of the
first 2 eigendirections, we get from \eref{eq:xhatm},
\begin{equa}
\ &  \XX{\P_0 M_{\phi_0,0}}=\XX{\P_0}\XX{M_{\phi_0,0}}\\
\ &  = \I 1_\phi\begin{pmatrix}[cc|c|c]
  0&0&&\\
  0&0&&\\
  \hline
  &&\one&\\
  \hline
  &&&\one\\
 \end{pmatrix} \begin{pmatrix}[cc|c|c]
     0&-2\I&&\\
  0&0&&\\
    \hline
    &&\one&\\
    \hline
  &&&-\one\\
\end{pmatrix}
=-\I 1_\phi \begin{pmatrix}[cc|c|c]
     0&0&&\\
  0&0&&\\
    \hline
    &&\one&\\
    \hline
  &&&-\one
  \end{pmatrix}~,
\end{equa}
where $1_\phi\equiv1+\phi$.
This is clearly the leading constant term.

The next term is at the origin of the $\epsilon $-splitting of
\tref{thm:imag}.
Using
\eref{eq:e11}, we get
\begin{equa}
\ &  \XX{\P_0 E_{11}}=\XX{\P_0} \XX{E_{11}}\\
\ &=\I\epsilon \begin{pmatrix}[cc|c|c]
  0&0&&\\
  0&0&&\\
  \hline
  &&\one&\\
  \hline
  &&&\one\\
\end{pmatrix}\begin{pmatrix}[cc|c|c]
    0&2\I&&\\
    0&0&&\\
    \hline
    &&S&\\
    \hline
    &&&-S
  \end{pmatrix}=\I\epsilon \begin{pmatrix}[cc|c|c]
    0&0&&\\
    0&0&&\\
    \hline
    &&S&\\
    \hline
    &&&-S
  \end{pmatrix}~.
\end{equa}
Thus, to leading order, we find
\begin{equ}\label{eq:blocks}
  \XX{\P_0 (M_{\phi_0,0}+E_{11})}=
  -\I 1_\phi \begin{pmatrix}[cc|c|c]
     0&0&&\\
  0&0&&\\
    \hline
    &&\one&\\
    \hline
  &&&-\one\\
  \end{pmatrix}+\I\epsilon \begin{pmatrix}[cc|c|c]
    0&0&&\\
    0&0&&\\
    \hline
    &&S&\\
    \hline
    &&&-S
  \end{pmatrix}~.
\end{equ}

We now use 
\begin{proposition}\label{prop:spectrum}
  Consider a tri-diagonal matrix $U$ with $U_{i,i+1}=U_{i,i-1}\ne0$
  for  all $i$ and
arbitrary
  elements in the diagonal.
  Then
  \begin{myenum}
    \item All eigenfunctions of $U$ have their first and last
      components non-zero
    \item All eigenvalues of $U$ are simple.
  \end{myenum}
\end{proposition}

Postponing the proof, we conclude, by applying the proposition to
$S$ and $-S$ separately, that:
\emph{The matrix $\XX{\P_0 E_{11}}$ has a double eigenvalue 0, and $2(n-1)$ simple,
  purely imaginary,
  eigenvalues $\pm\lambda _1,\dots,\pm\lambda _{n-1}$ which are
  different from 0.}
Because $S$ is symmetric, it follows from the proposition that
$\XX{\P_0(M_{\phi_0,0}+E_{11})}$ has purely imaginary spectrum, with two
eigenvalues equal to 0, and $n-1$ simple eigenvalues near $\I
(1+\phi)$ and another $n-1$ near $-\I (1+\phi)$. Furthermore, since
$E_{11}$ is proportional to $\epsilon $, and $S$ has simple
eigenvalues separated by $\OO(1)$, we conclude
\begin{corollary}\label{cor:split}
The
eigenvalues of $\XX{\P_0 (M_{\phi_0,0}+E_{11})}$ are purely imaginary and satisfy $|\mu_i-\mu_j|>C'_n \epsilon $ when $i\ne j$. The
constant $C'_n>0$ only depends on $n$.
The eigenfunctions, in the $\XX{\cdot}$ basis have non-zero component
at position $2+n-1$ and $2n$.
\end{corollary}

We will now show that the remaining terms of $\XX{\P M_{\phi,\epsilon
}}$ only give rise to corrections of order $\OO_2$, both in the
spectrum and in the eigendirections (on the subspace spanned by
$X^*(\xi,\eta)^\T$). For the terms which are of order $\OO_2$, there
is nothing to prove, since they perturb a matrix whose spectrum is
separated by $\OO(\epsilon )$.

But there are terms of order $\OO(\epsilon )$. They are of the form
$\XX{\P_2 M_{0,0}}$ or $\XX{\P_0 E_{12}}$. Here, the special form of the
  matrices comes into play,  and the following lemma formulates the
  key point:

\begin{lemma}
  Let $U$, $V$ be $(r+s) \times (r+s)$ matrices of the form 
  \begin{equ}
    U=
    \begin{pmatrix}[c|c]
      U_1&\\
      \hline
      &U_2
    \end{pmatrix}~,\quad
    V=
    \begin{pmatrix}[c|c]
      &V_1\\
      \hline
      V_2&
      \end{pmatrix}~,
  \end{equ}
where $U_1$ is an $r \times r$ square matrix, $U_2$ an $s \times s$
square matrix and $V_1$ and $V_2$ are
  $s \times r$ and $r \times s$ matrices, respectively.
  Let $x=(0,x_2)^\T$ be an eigenvector of $U$ (here $x_2 \in \real^s$). Then
  \begin{equ}\label{eq:xvxnull}
    \langle x | V x\rangle=0~.
  \end{equ}
\end{lemma}
\begin{proof}
  Obvious.
\end{proof}

We apply this lemma to the  two matrices { $\XX{\P_2 M_{0,0} }$ and
$\XX{\P_0 E_{12}}$,  }which play the role of $V$ in the lemma. From
\pref{prop:spectrum} we conclude that the eigenvectors of
\eref{eq:blocks} are of the form
\begin{equ}
  x_1=(0,0,*,\dots,*,0,\dots,0)^\T \text{  or }
x_2=(0,0,0,\dots,0,*,\dots,*)^\T~.
\end{equ}
Therefore \eref{eq:xvxnull} applies in this case, and thus the first
order contributions of $\XX{\P_2}$ resp.~$\XX{E_{12}}$ vanish. Thus,
as the spectra are  $\epsilon$-separated and simple by
Corollary~\ref{cor:split}, we see that for small enough $|\epsilon |$,
the spectrum maintains the splitting properties when the second order
perturbations (in $\epsilon $) are added. Note that, since
$\XX{M_{\phi_0,0}}=(1+\phi_0)\XX{M_{0,0}}$ by \eref{eq:mphi} and as
$\XX{M_{0,0}}$ has the form \eref{eq:xhatm}, the effect of $\phi_0$ is
to just shift the spectrum globally, without changing the spacing of
the eigenvalues within a block.
This completes the proof of  \tref{thm:imag}.
\end{proof}

We end the section with the
\begin{proof}[Proof of \pref{prop:spectrum}]
  Suppose $x=(x_1,\dots,x_n)^\T$ is an eigenfunction with eigenvalue
  $\lambda $. Then, from the form of $U$, we have
  \begin{equa}
    (U_{11}x_1+U_{12}x_2)&=\lambda x_1~, \text{ so that }
    x_2=\frac{1}{U_{12}}(\lambda -U_{11})x_1~,\\
    (U_{21}x_1+U_{22}x_2+U_{23}x_3)&=\lambda x_2~, \text{ so that }\\
    x_3&=\frac{1}{U_{23}} \bigl((\lambda -U_{22})x_2-U_{21}x_1 \bigr)\\
    &=\frac{1}{U_{23}}\left(
    (\lambda -U_{22})=\frac{1}{U_{12}}(\lambda -U_{11})x_1-U_{21}
    \right)x_1~.
  \end{equa}
  Continuing in this way, we can write $x_j=C_j x_1$ for some
  constant $C_j$ defined in terms of the $U_{ij}$ and $\lambda $. But
  then, if $x_1=0$ all other $x_j$ are 0, and we have not found a
  non-trivial eigenfunction. The same argument rules out the case
  $x_n=0$. Note that other $x_j$ can vanish.

  The proof of 2 is by the same argument: If we normalize, say,
  $x_1=1$, the inductive steps above show that the eigenfunctions are
  uniquely determined by the eigenvalues. Hence, since there is a
  complete set of eigenvectors, all eigenvalues are simple.
\end{proof}

\section{The effect of the dissipation on the semigroup}\label{sec:semigroup}

In order to control the evolution of $\zeta(t)$, we need to understand
precisely the bounds on the evolution of the semigroup generated by the 
linear part of the equation.  We find that all modes in the subspace orthogonal
to the zero eigendirections are uniformly contracted with a rate proportional
to $\gamma \epsilon$.  This is somewhat surprising due to the localized nature
of the dissipative term in the equation.  However, if follows from the
facts we have demonstrated above.  Namely, 
 we have shown in \tref{thm:imag} that the eigenvectors in the
$\XX{\cdot}$ representation have nonzero last component, and are
isolated, { and so we conclude from standard perturbation theory that, adding
dissipation $\Gamma$
moves these eigenvalues into the left half plane, by an amount
proportional to $\gamma\epsilon $ (up to higher order terms).  We now check
that the coefficients of the term proportional to $\gamma \epsilon$ are all non-zero (and depend only on $n$). }

Let
\begin{equ}\label{eq:ll}
{\LL_{\phi_0,\gamma }}={M_{\phi_0,\epsilon }} -{\Gamma}~,\quad\XX{\LL_{\phi_0,\gamma }}=\XX{M_{\phi_0,\epsilon }} -\XX{\Gamma}~.
\end{equ}
An explicit calculation shows that
\begin{equ}\label{eq:xxgamma}
  \XX{\Gamma}=
  \begin{pmatrix}[cc|ccc|ccc]
    0&0&&&&&&\\
    0&0&&&&&&\\
    \hline
    &&0&0&0&&&\\
    &&0&0&0&&&\\
    &&0&0&\gamma\epsilon &&&\\
    \hline
    &&&&&0&0&0\\
    &&&&&0&0&0\\
        &&&&&0&0&\gamma \epsilon 
  \end{pmatrix}~.
\end{equ}
\begin{proposition}\label{prop:semigroup}
  There is a constant $\gamma _0>0$, depending only on $n$ such that
  for $\gamma \in [0,\gamma _0]$ 
  one has the bound
  \begin{equ}\label{eq:semigroup}
    \left\|e^{t\LL_{\phi_0,\gamma}}X^*\zeta \right\|\le (1+C_n\gamma )
    e^{-\kappa_n\gamma \epsilon t}\| X^* \zeta\|~,
  \end{equ}
  for some $\kappa_n>0$ and $C_n>0$, depending only on $n$.
  \end{proposition}
\begin{proof}
  This result follows from the classical perturbation theory for
  eigenvalues and eigenfunctions.
  By \tref{thm:imag} we know that for $\gamma =0$ the purely
  imaginary eigenvalues are simple, and pairwise separated by
  $C\epsilon $, with $C>0$ depending only on $n$. 
  { We focus on the eigenvalues close to $+\I$ --- those near
  $-\I$ are handled by an entirely analogous procedure.  From
  \pref{prop:spectrum} we see that the eigenvectors $v_j$,
  $j=1,\dots,n-1$
  corresponding to
  these eigenvalues have a non-vanishing last component, and therefore
  there exists some $C'>0$, depending only on $n$, such that 
  $\langle v_j , \Gamma v_j\rangle>C' \gamma \epsilon $,  Therefore, up to
  higher order terms, standard perturbation theory for simple
  eigenvalues tells us that the spectrum of $\XX{\LL_{\phi,\gamma
  }}$ has (twice) $n-1$ eigenvalues in the negative halfplane at a
  distance of $\OO(\gamma \epsilon )$ from the imaginary axis.}

  We next show that the eigendirections make an angle of at most
  $\OO(\gamma )$ from the orthogonality of the eigendirections of the
  symmetric matrix $\XX{M_{\phi,\epsilon }}$, thus proving the bound
  \eref{eq:semigroup}. From perturbation theory (see \eg
  \cite{Kato1984}[I.$\mathsection$5.3]), the projection onto one of these
  eigenspaces is given by
  \begin{equ}
    P_j =-\frac{1}{2\pi \I}\int _{C_j} R(z) \d z~,
  \end{equ}
  where the contour $C_j$ is a circle of radius $\OO(\epsilon )$
  around the eigenvalue of the problem for $\gamma =0$ and $R$ is the
  resolvent.
   {
  The perturbed
  eigenvalue (which moves a distance $\sim -K_j \gamma\epsilon $ ) lies inside this circle,
  if $\gamma<\gamma_0 $ is sufficiently small, where $\gamma _0$
  depends only on $n$ but not on $\epsilon $, so long as $|\epsilon | < \epsilon_0$, for some fixed $\epsilon_0$. Therefore,
  $\norm{P_j}<1+\OO(\gamma )$, since the contour integral over the circle
  leads to a bound $1/\epsilon $ which cancels the factor $\epsilon $
  in $\gamma \epsilon $. }
\end{proof}
{
Note that since the change of variables matrix $X$ is orthogonal, these decay estimates also hold in the original coordinates, \ie 
\begin{corollary}\label{cor:semi}
  There is a constant $C_n$ such that 
  \begin{equ}\label{eq:semigroup_orig}
    \left\|e^{t\LL_{\phi_0,\gamma}} \zeta \right\|\le (1+C_n\gamma )
    e^{-\kappa_n\gamma \epsilon t}\|  \zeta\|~.
  \end{equ}
\end{corollary}

}

\section{Projecting onto the complement of the 0 eigenspace}\label{sec:estimate}

In this section, we reexamine equations \eqref{eq:3}-\eqref{eq:xieta2} to derive 
carefully, and explicitly, the equations for the evolution of the variables 
$\phi$, $\theta$, and $\zeta$.   In particular, we look at the constraints on
these equations imposed by the requirement that $\zeta$
remains in the range of $\P=\P_{\phi_0}$. As $\phi$ changes with time,
the projection will also generate terms involving $\phi(t)-\phi_0$.
We will bound these terms carefully, since they lead to secular growth
in $\zeta$.

As we will often have to compare $p(\phi(t))$ to $p(\phi_0)$,
it is useful to bound this difference as $\OO(\delta )$ with
\begin{equ}
  \delta =\phi(t)-\phi_0~.
\end{equ}
We will only be interested in small $\delta $.

We  fix a $\phi_0$ small enough for \tref{thm:imag} to apply.
We next analyze the terms on the r.h.s~of \eref{eq:xieta2}, one by one,
using that $\zeta$ is orthogonal to the $\n{j}_{\phi_0}$.
\begin{lemma}\label{lem:U}
  Consider the linear evolution operator
  \begin{equ}
    U=\pmat{0&-\bigl((L-\phi)
+(\pphi)^2\bigr)\\
(L-\phi) +3(\pphi)^2&0
}~.
  \end{equ}
  Then,
  \begin{equa}
\langle \n2_{\phi_0} |  U \zeta\rangle&=\OO(\delta)\|\zeta\|~,\label{eq:termU2}\\
\langle \n1_{\phi_0} |  U
\zeta\rangle&=\OO(\delta)\|\zeta\|~,\label{eq:termU1}\\
\P_{\phi_0} U\zeta &= U\zeta + \OO(\delta)\|\zeta\|~.\label{eq:termuP} 
  \end{equa}
\end{lemma}
\begin{proof}
Note that
\begin{equa}\Label{eq:ABX}
\left\langle\n2_{\phi_0}\mid U\zeta\right\rangle
=
\left\langle
U^*\n2_{\phi_0} \mid  \zeta\right\rangle~,
\end{equa}
and so,
\begin{equa}
U^*&\n2_{\phi_0}=\pmat{
0&\bigl((L-\phi) +3(\pphi)^2\bigr)\\
-\bigl(L-\phi
+(\pphi)^2\bigr)&0} \pmat{\pphizero\\0}=\pmat{\OO(\delta )\\0}~.
\end{equa}
We use here, and throughout, the smoothness of $p(\phi)$  and the expansion of $p(\phi)$. The replacement of $\phi$ by
$\phi_0$ therefore leads to an error term in \eref{eq:termU2} of the form
\begin{equa}
\OO(\delta )\norm\xi\le \OO(\delta ) \norm\zeta ~.
\end{equa}
This proves \eref{eq:termU2}.

We next study
$
\left\langle \n1_{\phi_0}\mid U\zeta\right\rangle
$
and take the adjoint
$\left\langle    U^*\n1_{\phi_0}\mid \zeta\right\rangle$, using\\
$\n1_{\phi_0}=(2+\OO(\epsilon+\phi_0 ))(0,\partial_{\phi_0} p(\phi_0))^\T$.
Recall  that
\begin{equa}
\pmat{0&B_{\phi_0}\\A_{\phi_0}&0}\n1_{\phi_0}=\n2_{\phi_0}~,
\end{equa}
which is orthogonal to $\zeta=(\xi,\eta)^\T$. We write
\begin{equ}
  U^* =\pmat{0&B_\phi-B_{\phi_0}\\A_\phi-A_{\phi_0}&0}+\pmat{0&B_{\phi_0}\\
A_{\phi_0}&0}~,
\end{equ}
and therefore
\begin{equa}
\left\langle \n1_{\phi_0}\mid U\zeta\right\rangle&=
\left\langle \pmat{0&B_\phi-B_{\phi_0}\\
A_\phi-A_{\phi_0}&0}\n1_{\phi_0}\mid \zeta\right\rangle
&+\left\langle  \pmat{0&B_{\phi_0}\\
A_{\phi_0}&0}\n1_{\phi_0}\mid \zeta\  \right\rangle\\
&= \OO(\delta) \|\zeta \|~,
\end{equa}
since the second term is zero by construction.
This proves \eref{eq:termU1}.
The identity \eref{eq:termuP} follows.
\end{proof}

\begin{lemma}\label{lem:gamma}
The $\Gamma$-dependent terms of \eref{eq:xieta2} lead to the bounds
\begin{equa}
\left\langle \n2_{\phi_0} \mid \pmat{-C_\Gamma\xi-C_\Gamma\pphi\\-C_\Gamma\eta}\right\rangle&=\OO(\gamma \epsilon
  ^{n})\norm\zeta-2(1+\OO(\delta))\gamma \epsilon ^{2n-1}~,\label{eq:t2}\\
\left\langle \n1_{\phi_0}\mid \pmat{-C_\Gamma\xi-C_\Gamma\pphi\\-C_\Gamma\eta}
\right\rangle
&=\OO (\gamma \epsilon ^{n})\norm\zeta ~,\label{eq:t2a}\\
{\| \P_{\phi_0} \Gamma\zeta  - \Gamma\zeta \| }& { \le 
C \gamma\epsilon^{n} \norm\zeta }  ~.\label{eq:t2P}
\end{equa}

\end{lemma}

\begin{proof}From the definition of $\n2_{\phi_0}$ we get
\begin{equa}
  2&\left\langle
\pphizerozero
\mid\pmat{-C_\Gamma\xi-C_\Gamma\pphi\\-C_\Gamma\eta}
 \right\rangle\\&= -2\gamma\epsilon \xi_n\, \pphizero_n -2\gamma \epsilon
 (p(\phi_0))_n (p(\phi))_n\\&=\OO(1)
{\gamma\epsilon  \xi_n\epsilon ^{n-1}}-2\gamma \epsilon
^{2n-1}(1+\OO(\delta))~,
\end{equa}
using the expansion of $p(\phi)$ in powers of $\epsilon $, and
observing that $C_\Gamma$ is proportional to $\gamma \epsilon $.

Similarly, from the definition of $\n1_{\phi_0}$, we get
\begin{equa}
\left\langle \n1_{\phi_0}\mid \pmat{C_\Gamma\xi+C_\Gamma\pphi\\C_\Gamma\eta}
\right\rangle
=  \partial_{\phi_0}p(\phi_0)\cdot C_\Gamma\eta=\gamma \epsilon (\n1_{\phi_0})_n
\eta_n=\OO(1) \gamma \epsilon ^{n}\eta_n~,
\end{equa}
using the expansion for $(\n1_{\phi_0})_j=\partial_{\phi_0}p_j(\phi_0)=
\OO( \epsilon ^{j-1})$.
The last equation follows from the first two.
\end{proof}

\

\begin{lemma}\label{lem:tphitheta}
  Consider the terms involving $\tdotphitheta$. We have, omitting
  throughout the
  factor $\tdotphitheta$:
\begin{equa}
 \left\langle \n2_{\phi_0} \mid \pmat{\eta\\-\xi-p(\phi)} \right\rangle
&= (2+\OO(\epsilon+\phi ))\norm\zeta ~,\label{eq:t3}\\
 \left\langle \n1_{\phi_0} \mid \pmat{\eta\\-\xi-p(\phi)} \right\rangle
&= -1+\OO(\delta+\epsilon )+\OO(1)\norm\zeta ~,\label{eq:t3b}\\
 \P_{\phi_0} \pmat{\eta\\-\xi-p(\phi)} 
&=\pmat{\OO(\norm\zeta )\\\OO(\delta+\epsilon+\norm\zeta)}~.\label{eq:t3c}\\
\end{equa}

\end{lemma}

\begin{proof}

  \eref{eq:t3} follows by observing that
\begin{equa}
 \left\langle \pphizerozero \mid \pmat{\eta\\-\xi} \right\rangle
= (\pphizero\cdot \eta)~,
\end{equa}
and
\begin{equa}
 \left\langle \pphizerozero \mid \pmat{0\\-\pphi}
\right\rangle=0~,
\end{equa}
and using $\norm{p(\phi_0)}=1+\OO(\epsilon+\phi )$.
To prove \eref{eq:t3b}, observe that, by our normalization,
\begin{equa}
  \left\langle \n1_{\phi_0} \mid  \pmat{0\\\pphizero}
  \right\rangle=\left\langle \n1_{\phi_0}\mid \v1_{\phi_0} \right\rangle=1~,
\end{equa}
and therefore,
\begin{equa}
 \left\langle \n1_{\phi_0} \mid  \pmat{0\\-\pphi}\right\rangle
  =-1+\OO(\delta)~.
\end{equa}
On the  other hand,
\begin{equa}
  \left\langle \n1_{\phi_0} \mid \pmat{\eta\\-\xi}
  \right\rangle=-\partial_{\phi_0}\pphizero
\xi~,
\end{equa}
and thus \eref{eq:t3b} follows.
Finally \eref{eq:t3c} follows from \eref{eq:t3} and \eref{eq:t3b};
\begin{equa}
  \P_{\phi_0}& \pmat{\eta\\-\xi-p(\phi)}=\pmat{\eta\\-\xi-p(\phi)}\\&-
  (2+\OO(\epsilon ))\pmat{\partial_{\phi_0}p(\phi_0)\\0}(p(\phi_0)\cdot \eta)
  \\&-\pmat{0\\p(\phi_0)}(-1+\OO(\delta+\epsilon+\norm\zeta  ))~.
\end{equa}
The term involving $\eta$ cancels by the normalization of $\v2$ and $\n2$.
The term $-p(\phi_0)\cdot (-1)$ cancels with $-p(\phi)$ up to
$\OO(\delta+\epsilon  )$, and thus, \eref{eq:t3c} follows.
\end{proof}

\begin{lemma}\label{lem:phi}
  The terms involving $\dot \phi$ are bounded as follows (omitting the
  factor $\dot \phi$):
  \begin{equa}
   \left\langle \n2_{\phi_0} \mid
  \pmat{-\dpphi\\0}\right\rangle
  &=-1+\OO(\delta)~,
  \label{eq:t4}\\
   \left\langle \n1_{\phi_0} \mid
  \pmat{-\dpphi\\0}\right\rangle
  &=0~,
  \label{eq:t4b}\\
  \P_{\phi_0} \pmat{-\partial_{\phi}p(\phi)\\0}&=
  \pmat{\OO(\delta  )\\0}~.\label{eq:t4c}
     \end{equa}
\end{lemma}
\begin{proof}
  Recall that
  \begin{equ}
    \pmat{\dpphi\\0}\sim (1/2,0,\dots,0)^\T~,
  \end{equ}
  and so, since $\n2_{\phi_0}\sim 2 (p(\phi_0),0)^\T$, we find
  \begin{equa}
\left\langle \n2_{\phi_0} \mid
  \pmat{-\dpphi\\0}\right\rangle&=\left\langle 2\pphizerozero\mid
\pmat{-\dpphi\\0}\right\rangle\\
&=- 2\pphizero\cdot \partial_\phi\pphi= -1+\OO(\delta )~,
  \end{equa}
  which is \eref{eq:t4}. From the form of $\n1$, \eref{eq:t4b} is
  obvious.
  Finally,
  \begin{equa}\Label{eq:bizarre}
    \P_{\phi_0} \pmat{-\partial_{\phi}p(\phi)\\0}&=
    \pmat{-\partial_{\phi}p(\phi)\\0}+\pmat{\partial_{\phi_0}p(\phi_0)\\0}
    +\pmat{\OO(\delta  )\\0}~.
  \end{equa}
  
\end{proof}

We now combine the lemmas \ref{lem:U}--\ref{lem:phi}.
Note that $\Q_{\phi_0}\equiv \one-\P_{\phi_0}$ projects on a
two-dimensional space. 
Let $\Q^{(j)}=|\v{j}_{\phi_0}\rangle \langle \n{j}_{\phi_0} |$,
$j=1,2$.

For $j=2$, we get contributions:\\
From \eref{eq:termU2}, we have $\Q^{(2)} U\zeta =\OO(\delta)
\norm{\zeta }$.\\
From \eref{eq:t2} we get $\Q^{(2)}\pmat{C_\Gamma
  (\xi+p(\phi))\\C_\Gamma \eta}=-2(1+\OO(\delta))\gamma \epsilon ^{2n-1}+\OO(\gamma \epsilon
^n )\norm\zeta $.\\
From \eref{eq:t3} we get $\tdotphitheta\Q^{(2)}\pmat{\eta\\-\xi-p(\phi)}=\tdotphitheta\OO(\norm\zeta)$, and\\ from
\eref{eq:t4} we get
$\dot\phi\Q^{(2)}\pmat{-\partial_\phi
  p(\phi)\\0}=\dot\phi(-1+\OO(\delta  ))$.

Similarly,
for $j=1$, we get contributions:\\
From \eref{eq:termU1}, we have $\Q^{(1)} U\zeta =\OO(\delta )
\norm{\zeta }$.\\
From \eref{eq:t2a} we get $\Q^{(1)} \pmat{C_\Gamma
  (\xi+p(\phi))\\C_\Gamma \eta}=\OO(\gamma \epsilon
^n) \norm\zeta $.\\
From \eref{eq:t3b} we get\\
$\tdotphitheta\Q^{(1)}\pmat{\eta\\-\xi-p(\phi)}=\tdotphitheta(-1+\OO(\delta
+\epsilon +\norm\zeta)$, and\\ from
\eref{eq:t4b} we get
$\dot\phi\Q^{(1)}\pmat{-\partial_\phi p(\phi)\\0}=0$.

By construction we have that $\Q_{\phi_0}\equiv \one-\P_{\phi_0}$,
projects on the null-space, and therefore
\begin{equa}
  \Q_{\phi_0} \dot \zeta =0~.
\end{equa}
Since $\Q_{\phi_0}\dot\zeta=0$, we find, upon summing, for the ``1'' component (and recalling the nonlinear terms in \eqref{eq:xieta2})
\begin{equa}\label{eq:component1}
  0&=\OO(\delta )\norm{\zeta }+(-2+\OO(\delta ))\gamma \epsilon ^{2n-1}\\&+\OO(\gamma \epsilon
^n \norm\zeta)+ \tdotphitheta\OO(\norm\zeta)+\dot\phi(-1+\OO(\delta )) + {\OO(\| \zeta \|^2)}~,
\end{equa}
and for the ``2'' component:
\begin{equa}\label{eq:component2}
  0&= (\OO(\delta )+\OO(\gamma \epsilon
^n) \norm\zeta- \tdotphitheta(1+\OO(\delta +\epsilon +\norm\zeta ))+ {\OO(\| \zeta \|^2)}~.
\end{equa}

Finally, using the projection $\P_{\phi_0}$, we find:
From \eref{eq:termuP},\\ $\P_{\phi_0} U\zeta = U\zeta +
\OO(\delta )\|\zeta\|$,\\
from \eref{eq:t2P}, $\P_{\phi_0}\Gamma\zeta =\Gamma\zeta
+\OO(\gamma\epsilon^{n} )\|\zeta\|$,\\
from \eref{eq:t3c}, $\tdotphitheta\P_{\phi_0} \pmat{\eta\\-\xi-p(\phi)} 
=\tdotphitheta\pmat{0\\\OO(\delta+\epsilon )}$,\\
and finally from \eref{eq:t4c},
$ \dot\phi\P_{\phi_0} \pmat{-\partial_{\phi}p(\phi)\\0}=
  \dot \phi\pmat{\OO(\delta )
    \\0}$.

  Summing these terms, 
  we get {
  \begin{equa}\label{eq:componentzeta}
    \dot \zeta =\LL_{\phi_0,\gamma}\zeta  +\tdotphitheta
    \pmat{0\\\OO(\delta+°\epsilon )} +\dot \phi\pmat{\OO(\delta )
    \\0} + \OO(\| \zeta \|^2) ~.
  \end{equa}
}

  Simplifying the notation somewhat, and substituting
  \eref{eq:component2} into \eref{eq:component1} we formulate
  \eref{eq:component1}-(\ref{eq:componentzeta}) as a proposition: {
  \begin{proposition}\label{prop:evolution}
    One has
    \begin{equa}[eq:zetadot]
      t\dot\phi +\dot \theta &=\OO(\delta +\gamma \epsilon ^{n})\norm\zeta  + \OO(\| \zeta \|^2)~,\Label{eq:tphidot}\\
  \dot\phi &=-(2+\OO(\delta ))\gamma\epsilon ^{2n-1}+\OO(\delta
  +\gamma \epsilon ^n)\norm\zeta + \OO(\| \zeta \|^2)~,\Label{eq:phidot}\\
  \dot\zeta&= \LL_{\phi_0,\gamma} \zeta +\OO(\delta +\gamma \epsilon
  ^n)\norm\zeta\pmat{0\\\OO(\delta +\epsilon )}\\&~~~-\pmat{(2+\OO(\delta ))\gamma\epsilon
    ^{2n-1}\OO(\delta )\\0} + \OO(\| \zeta \|^2)~.
\end{equa}
Here, as $\zeta=(\xi,\eta)^\T$,
\begin{equa}\Label{eq:generator}
  \LL_ {\phi_0,\gamma}\zeta= 
  \begin{pmatrix}
    0&A_{\phi_0}\\
    B_{\phi_0} &0
  \end{pmatrix} \begin{pmatrix}
    \xi\\\eta
  \end{pmatrix} - 
  \begin{pmatrix}
    C_\Gamma &0\\
    0&C_\Gamma
  \end{pmatrix}
  \begin{pmatrix}
    \xi\\\eta
  \end{pmatrix}~.
\end{equa}
  \end{proposition}
}

\section{Bounds on the evolution of $\zeta $}\label{sec:zeta}

In principle, $\|\zeta\|$ can grow as the system evolves, and there
are { two }possible causes. First, for short times, the bound
\eref{eq:semigroup_orig}
\begin{equ}
    \left\|e^{t\LL_{\phi_0,\gamma}} \zeta \right\|\le (1+C_n\gamma )
    e^{-\kappa_n\gamma \epsilon t}\|  \zeta\|~,
  \end{equ}
does not contract. {Secondly, $\zeta(t)$ is orthogonal to the 
cylinder of breathers at the {\em initial} point $(\phi_0,\theta_0)$, but as
$\phi$ and $\theta$ evolve with time, this is no longer the case.
We must periodically reorthogonalize $\zeta(t)$ by a
procedure which we detail in the next
section, and which replaces $\zeta (T)$ by $\widehat\zeta$. }
As we will show in \eref{eq:reorthog2}, this leads to
a growth which is bounded by
\begin{equ}\label{eq:repeat}
\| \widehat{\zeta} \| \le (1+ C_R (\phi(T)-\phi_0) \gamma \epsilon^n) \| \zeta(T)\|\ .
\end{equ}

In this section we show that the contraction in the semigroup generated by the dissipative terms in the equation
is sufficient to overcome those growths, if we wait a sufficiently long time. We
will show (up to details spelled out below) that if $\norm{\zeta
  (0)}\le \gamma \epsilon ^n$, and $\TT=C_\TT/\epsilon $, then $(1+C_R
|\phi(\TT)-\phi_0|\gamma \epsilon ^n)\norm{\zeta (\TT)}\le
\gamma \epsilon ^n$. Furthermore for all $t\in[0,\TT]$ one has $\norm{\zeta
  (t)}\le 4 \gamma \epsilon ^n$.
To prove such statements,
we reconsider the equations of \pref{prop:evolution} which
we rewrite {in a slightly simplified way:}
We define
$$\delta(t)
=\phi(t)-\phi_0~,
$$ and then
\begin{equa}
  \dot s &=\OO(\delta +\gamma \epsilon ^{n})\norm\zeta + \OO( \| \zeta(t) \|^2) ~,\label{eq:tphidot2}\\
  \dot\delta  &=-(2+\OO(\delta ))\gamma\epsilon ^{2n-1}+\OO(\delta
  +\gamma \epsilon ^n)\norm\zeta +  \OO( \| \zeta(t) \|^2) ~ ~,\label{eq:phidot2}\\
    \dot \zeta &=\LL_{\phi_0,\gamma}\zeta  +\dot s
    \pmat{0\\\OO(\delta+°\epsilon )} +\dot \delta \pmat{\OO(\delta )
    \\0} +  \OO( \| \zeta(t) \|^2) ~ ~.\label{eq:zetadot2}
  \end{equa}
Here, as $\zeta=(\xi,\eta)^\T$,
\begin{equa}\Label{eq:generator2}
  \LL_ {\phi_0,\gamma}\zeta= 
  \begin{pmatrix}
    0&A_{\phi_0}\\
    B_{\phi_0} &0
  \end{pmatrix} \begin{pmatrix}
    \xi\\\eta
  \end{pmatrix} - 
  \begin{pmatrix}
    C_\Gamma &0\\
    0&C_\Gamma
  \end{pmatrix}
  \begin{pmatrix}
    \xi\\\eta
  \end{pmatrix}~.
\end{equa}
\begin{definition}
  We define the arrival time $\TT$ by
  \begin{equ}\label{eq:arrival}
    \TT= \frac{8C_n}{\kappa_n \epsilon }\equiv  C_\TT \epsilon ^{-1} ~.
  \end{equ}
  This definition ensures that
  \begin{equ}\label{eq:cngamma}
    (1+\frac{3}{2} C_n\gamma )e^{-\kappa_n\gamma  \epsilon \TT/4} \le 1~.
  \end{equ}
  The remaining factor $e^{-\kappa_n \gamma \epsilon \TT/4}$ will be used to
  bound $ (1+C_R|\delta |\gamma \epsilon ^n)$, while another $e^{-\kappa_n \gamma \epsilon
    \TT/2}$
  will be used to bound the contributions from the mixed terms in \eref{eq:tphidot2}--(\ref{eq:zetadot2}).
\end{definition}
Since we have a coupled system, we introduce a norm over times in $[0,\TT  ]$.
Let $x=(s,\delta ,\zeta)$,  and consider a family of functions
\begin{equ}
  \{x\}_{\TT  }=\{x(\tau )\}_{\tau \in[0,\TT  ]}~.
\end{equ}
We define
\begin{equa}
  \nnorm{\{x\}_t} =\max \bigl( \sup_{\tau \in[0,t]} |s(\tau )|,
  \sup_{\tau \in[0,t] }|\delta (\tau )| ,C_\zeta \norm{\zeta(\tau
    )}\bigr)~,\quad\text{with}\quad
  C_\zeta^{-1}  =\gamma \epsilon^{n}~.
\end{equa}
The equations \eref{eq:tphidot2}--(\ref{eq:zetadot2}) define an
evolution $t\mapsto \FF^t$.

\begin{theorem}\label{thm:evolution}
  Let $t\mapsto x(t)$ be a family of functions (not necessarily a
  solution of the system above) which satisfies
  $\nnorm{\{x\}_{\TT  }}\le 2$. Define $\FF x=\{\FF^\tau  x\}_{\tau
    \in[0,\TT  ]}$ as the family evolving from $x(0)$.
  If $\nnorm{\{x\}_0}\le2$ and $s(0)=\delta (0)=0$, then
the solution of the system above satisfies
\begin{equ}
  \nnorm{\{\FF  x\}_{t}} \le 4~, 
\end{equ}
 for all $t\in[0,\TT  ]$. In other words, $\FF$ maps such initial
 conditions into the sphere of
 radius $4$.

 Furthermore, when $\nnorm{\{x\}_0}\le 1$ (this means in particular $\norm{\zeta
   (0)}\le \gamma \epsilon ^n$) then the solution of the
 above system satisfies at time $\TT$:\footnote{We are not claiming
   that such bounds hold for all $t\le \TT$. The effect of the
   dissipation needs time to set in (at least if we want to re-project
   onto a new axis after some time).}
 \begin{equa}
  | \delta (\TT  )& + 2\gamma \epsilon ^{2n-1} \TT  | \le 2\gamma \epsilon ^{2n-1/2} \TT ~,\\
   \norm{\zeta (\TT  )} &\le e^{- \kappa_n \gamma \epsilon
     T/4} e^{- \kappa_n \gamma \epsilon
     T/2}(1-\gamma \epsilon)^{-1} (1+ \frac{3}{2} C_n \gamma) \gamma
   \epsilon^n \\&\le e^{- \kappa_n \gamma \epsilon
     T/4}\gamma \epsilon ^n~.
 \end{equa}
\end{theorem}

\begin{corollary}\label{cor:project}
  Referring to \eref{eq:repeat} (\ie \eref{eq:reorthog2}), we get the
  bound
  \begin{equa}
  \|\widehat\zeta\| &\le(1+ C_R (\phi(T)-\phi_0) \gamma \epsilon^n) \|
  \zeta(T)\|\\&\le   e^{- \kappa_n \gamma \epsilon T/4}(1+ C_R
  (\phi(T)-\phi_0) \gamma \epsilon^n)\gamma \epsilon ^n\le \gamma
  \epsilon ^n~.
  \end{equa}
 In other words, $\|\widehat\zeta\|$ (at time $T$) stays within the region $\gamma
 \epsilon ^n$.
\end{corollary}

\begin{remark}
The norm $\nnorm{\cdot}$ was introduced to allow for an a priori bound
on $\zeta (t)$ which is needed because of our way to estimate the
evolution of the coupled system \eref{eq:tphidot2}--(\ref{eq:zetadot2}).
\end{remark}
\begin{remark}
  We assumed $\delta (0)=0$ since that is the case which interests
  us. Also, by the gauge invariance, we may assume that { $\theta_0=\theta(0)=0$.}
\end{remark}

\begin{proof}

{

  We first study $\delta $.
  \begin{lemma}\label{lem:delta}
    Assume $\nnorm{\{x\}_{\TT  }}\le 2$. Then, we
    have, for $t\le \TT  =C_{\TT  }/\epsilon $,
    \begin{equ}\label{eq:deltabound}
    |   \delta (t) + 2\gamma \epsilon ^{2n-1}t | \le   2 \gamma \epsilon^{2n-1/2} t~.
    \end{equ}
  \end{lemma}
  
  \begin{remark} Note that this means that to lowest order in $\epsilon$, 
  $\delta(t) \sim 2 \gamma \epsilon^{2n-1} t$ for $0 \le t \le T$,
    which is the rate we found in \cite{EckmannWayne2018}.
  \end{remark}
  
\begin{proof}[Proof of \lref{lem:delta}]It is here that we use the a
  priori bound, and later we will see that the actual orbit of
  $\zeta(\tau )$ indeed satisfies  this bound.  
By the assumption, we have $\norm{\zeta(\tau )}\le 2/C_\zeta=2 \gamma \epsilon
^{n}$ for $\tau \in[0,\TT  ]$. Therefore, we can bound $\delta (\tau )$ as
follows:
The equation (\ref{eq:phidot2}) is of the form (with local names) and
finite constants $C_B$ and $C_C$:
\begin{equa}
  \dot \delta (t)&= -A +B(t)\delta (t)+C(t)~,\\
  A~~&=2\gamma \epsilon ^{2n-1}~,\\
  |B(t)|&\le C_B(\gamma\epsilon
  ^{2n-1}+\norm{\zeta(t)})~,\\
  |C(t)|&\le C_C(\gamma \epsilon ^n\norm{\zeta(t)} + \| \zeta(t) \|^2)~.
\end{equa}
We have  $\delta (0)=0$.
  The equation for $u(t)\equiv \delta (t)+At$ reads
  \begin{equa}\label{eq:uequ}
    \dot u = B(t)u(t) + (C(t)-At\cdot B(t))~.
  \end{equa}

    \begin{lemma}\label{lem:u}
    If $\nnorm{\{x\}_{\TT  }}\le 2$, then
      \begin{equ}
         |u(t)| \le At \epsilon ^{1/2} ~.
      \end{equ}
  \end{lemma}
    This clearly proves \eref{eq:deltabound} and hence \lref{lem:delta}.
    \end{proof}

\begin{proof}[Proof of \lref{lem:u}]

  The solution of \eref{eq:uequ} is
  \begin{equa}\label{eq:du}
    u(t)=\int_0^t\d\tau \,\bigl( C(\tau )-A\tau B(\tau )\bigr)e^{\int_\tau ^t \d \tau ' B(\tau ')}~.
  \end{equa}
  Let $
  B_{\max}=\max_{\tau \in[0,t]} |B(\tau )|$, and
  $C_{\max}=\max_{\tau \in[0,t]} |C(\tau )|$.
  From the assumptions, we have, for sufficiently small $\epsilon $,
  \begin{equa}
    B_{\max}t &\le C_{\TT  }C_B \K(\gamma \epsilon ^{2n-1}+C_\zeta
    ^{-1})/\epsilon\\
    &\le C_{\TT  }C_B \K(\gamma \epsilon ^{2n-2}+\epsilon ^{n-1/2})\ll
    \K\epsilon ~,\\
    C_{\max} &\le C_C \K(\gamma \epsilon^n C_\zeta ^{-1} + 4 C_{\zeta}^{-2} )\le 4 C_C  \gamma
    \epsilon^ {2n}~,
  \end{equa}
  (where we have assumed that $\gamma < 1/2$).
  The term coming from $C(\tau )$ in \eref{eq:du} is bounded by
  \begin{equa}\Label{eq:c1}
     C_{\max} \frac{|1-e^{B_{\max}t}|}{B_{\max}}\le 2C_{\max}t~,
  \end{equa}
  since  $    B_{\max}t\ll 1$.

  This leads to a bound for $2C_{\max}t$ of the form
  \begin{equa}
  2C_{\max}t\le 8 C_C\gamma \epsilon ^{2n}t~,
  \end{equa}
  which is much smaller than $At=2\gamma \epsilon ^{2n-1}t$ (when
  $\epsilon $ is small enough).
  The term in the integral coming from $At\cdot B(t)$ can be
  bounded as:
    \begin{equa}
    \int_0^t \d\tau  \,& A\tau \cdot B_{\max} e^{(t-\tau )B_{\max} }=
    A\frac{|B_{\max} t -e^{B_{\max}t}+1|}{B_{\max}}\\&\le At\left
    (\frac{B_{\max}t}{2}+\OO((B_{\max}t)^2)\right)\le At B_{\max}t~,
  \end{equa}
 since we already showed $    B_{\max}t\ll \K \epsilon $.
 Collecting terms, we get, for
 $t\le C_{\TT  }/\epsilon $,
 \begin{equ}
   |u(t)| \le (8C_C\gamma \epsilon^{2n} +2 \epsilon A )t \le A t \epsilon^{1/2}~,
 \end{equ}
which completes the proof of \lref{lem:u}.
\end{proof}

}

We continue the proof of \tref{thm:evolution}. The evolution of $s$ is
bounded in much the same way as that of $\delta $, and this is left to
the reader. (We actually do not make use of these bounds.)
We finally analyze the evolution of  $\zeta $, \eref{eq:zetadot2},
which controls the motion of the distance from the cylinder. By the estimates
\eref{eq:phidot2} and \eref{eq:deltabound} on $\delta $
and $\dot\delta $, we see that
\begin{equ}
\OO(\delta \dot\delta)=\OO(\gamma^2
\epsilon ^{4n-2}t)+\OO(\gamma \epsilon ^{2n-1}t+\gamma \epsilon ^{n}) \norm\zeta ~.
\end{equ}
{
Therefore, the equation for $\dot \zeta $ takes the form
\begin{equa}\label{eq:more}
  \dot\zeta = \LL_{\phi,\gamma } \zeta +\OO(\delta +\epsilon
  )\cdot\OO(\delta +\gamma \epsilon ^n)\norm\zeta+\OO(\| \zeta \|^2) + \OO(\gamma ^2 \epsilon ^{4n-2}t)~.
\end{equa}
Using \eref{eq:deltabound}, this simplifies to
\begin{equa}
  \dot\zeta = \LL_{\phi,\gamma } \zeta +\OO( \gamma \epsilon
  ^{n+1}+\gamma \epsilon ^{2n-1}t )\norm\zeta+\OO(\| \zeta \|^2)+\OO(\gamma ^2\epsilon ^{4n-2}t)~.
\end{equa}
From the estimates on the semigroup generated by $\LL_{\phi,\gamma}$ from Proposition \ref{prop:semigroup}
we conclude that
\begin{eqnarray}\nonumber
  \norm{\zeta (t)}\le (1+C_n \gamma )e^{-\kappa_n\gamma \epsilon
    t/2}\norm{\zeta _0} &+& R_2  (1+C_n\gamma )\int_0^t e^{-\kappa_n \gamma \epsilon (t-s)/2} \| \zeta(s) \|^2 ds
 \\  \label{eq:zetasimple}
 &&  + (1+C_n\gamma )\int_0^t \d\tau \, e^{-\kappa_n\gamma \epsilon \tau/2
  }X~,
\end{eqnarray}
where $X=\OO(\gamma^2\epsilon ^{4n-2}(t-\tau ))$ bounds the contribution from
the last term in \eref{eq:more}. We note that the contribution from
$\OO(\gamma \epsilon ^{n+1} +\gamma \epsilon ^{2n-1}t)$ which also multiplies
$\zeta $, has been absorbed 
into half the decay rate $\kappa_n\gamma \epsilon $.

Define 
\begin{equ}
Z(t) = \sup_{0 \le \tau \le t} e^{ \kappa_n\gamma \epsilon \tau/2} \| \zeta(\tau) \|\ .
\end{equ}
Then, from \eqref{eq:zetasimple}, we see that
\begin{equa}
Z(t) &\le  (1+ C_n \gamma) \| \zeta_0 \|\\&\quad + R_2 (1+ C_n \gamma) \int_0^t e^{- \kappa_n \gamma \epsilon s} ds
(Z(t))^2 +  C_X
(  \gamma^2 \epsilon ^{4n-2}t^2 e^{ \kappa_n\gamma \epsilon t /2} )
\\
& \le (1+ C_n \gamma) \| \zeta_0 \| + \frac{2 R_2 (1+ C_n \gamma)}{\kappa_n \gamma \epsilon}
(e^{\kappa_n \gamma \epsilon t/2} -1) (Z(t))^2 \\ \nonumber
& \quad +  C_X
(  \gamma^2 \epsilon ^{4n-2}t^2 e^{ \kappa_n\gamma \epsilon t /2} )\ .
\end{equa}
Suppose that $\| \zeta_0 \| \le \gamma \epsilon^n$.  Then, by continuity, for $t$ small, we have
\begin{equ}
\frac{2 R_2 (1+ C_n \gamma)}{\kappa_n \gamma \epsilon } Z(t) \le \gamma \epsilon\ .
\end{equ}
Define $T^*$ to be the largest value such that
\begin{equ}
\sup_{0 \le t \le T^*} \frac{2 R_2 (1+ C_n \gamma)}{\kappa_n \gamma \epsilon } Z(t) \le \gamma \epsilon\ .
\end{equ}
Then, 
\begin{equ}
\bigl(1 - \frac{2 R_2 (1+ C_n \gamma)}{\kappa_n \gamma \epsilon } Z(t)\bigr) Z(t) 
\le (1+ C_n \gamma) \| \zeta_0 \| +  C_X
(  \gamma^2 \epsilon ^{4n-2}t^2 e^{ \kappa_n\gamma \epsilon t /2} )\ ,
\end{equ}
or
\begin{equ}
Z(t) \le (1-\gamma \epsilon)^{-1} \left[ (1+ C_n \gamma) \| \zeta_0 \| +  C_X
(  \gamma^2 \epsilon ^{4n-2}t^2 e^{ \kappa_n\gamma \epsilon t /2} ) \right]\ ,
\end{equ}
for $0 \le t \le T^*$.

Since $T = \frac{8 C_n}{\kappa \epsilon}$,  if $n > 2$,
and if $\epsilon $ is sufficiently small, then $T \le T^*$ and we have for $0 \le t \le T$,
\begin{equ}
Z(t) \le (1-\gamma \epsilon)^{-1}  (1+ \frac{3}{2} C_n \gamma) \gamma \epsilon^n\ .
\end{equ}

From the definition of $Z(t)$, this implies
\begin{equ}
\| \zeta(t) \| \le e^{- \kappa_n \gamma \epsilon t/2} (1-\gamma \epsilon)^{-1}  (1+ \frac{3}{2} C_n \gamma) \gamma \epsilon^n\ ,
\end{equ}
or 
\begin{equ}
\| \zeta(T) \| \le \gamma \epsilon^n\ ,
\end{equ}
using the definition of $T$.
}

\end{proof}

\section{Re-orthogonalization}\label{sec:reorthog}

As the system evolves, the solution will remain close (at least for some time) to the 
cylinder of breather solutions for the $\gamma = 0$ equations.  However, it will drift, so
that the base point on the cylinder changes with time, while the
vector $\zeta $ stays orthogonal to the tangent space to the cylinder at the initial base point.
In the first subsection we show that we can periodically choose new coordinates
in such a way that $\zeta$ remains small for a very long time, while the frequency
$\phi$ of the base point in the cylinder changes in a controlled and computable way.
The change in the base point manifests itself in the presence of the terms proportional to $\delta$ in the equation for $\zeta$.
To counteract this secular growth, we will  stop the evolution after a long, but
finite, interval  and ``reset'' the initial 
data so that the ``new'' initial data $\widehat{\zeta}$ is again orthogonal to the tangent space at the ``new''
initial point $(\widehat{\phi},\widehat{\theta})$ on the cylinder.  Our approach in this section is inspired by the work of Promislow \cite{promislow:2002} on pattern formation in reaction-diffusion equations, but is complicated by the very weak dissipative properties of the semigroup $e^{t \LL_{\phi,\gamma }}$.
In particular, we will not be able to show that the normal component, $\zeta$ of the solution is strongly contracted, but
we will prove that it remains small for a very long period, during which the solution evolves close to the cylinder of
breathers.

Key to this approach is the fact that in a sufficiently small neighborhood of the cylinder of breathers, the angle and phase of the point on the cylinder
and the normal direction at that point provide a smooth coordinate system.  More precisely, one has:
\begin{proposition}\label{prop:IFT}  Fix $0 < \Phi_0 \ll 1$.  There exists $\mu > 0$ such that for any 
$\bar{\phi} \in [1-\Phi_0,1+\Phi_0]$, $\bar{\theta }\in [0,2\pi)$, $\| \bar{\zeta} \| < \mu$, there exists
$(\widehat{\phi}, \widehat{\theta}, \widehat{\zeta})$ such that
\begin{equ}
e^{\I\bar{\theta}} p(\bar{\phi}) + \bar{\zeta} = e^{\I\widehat{\theta}} p(\widehat{\phi}) + \widehat{\zeta}\ ,
\end{equ}
and $\widehat{\zeta}$ is normal to the tangent space of the family of breathers at $(\widehat{\phi},\widehat{\theta})$.
\end{proposition}

\begin{remark}  The utility of this proposition is that if we choose any point near our family of breathers, we can find $(\widehat{\phi},\widehat{\theta},\widehat{\zeta})$
to use as initial conditions for our modulation equations \eqref{eq:zetadot} with $\widehat{\zeta} \in {\mathrm{Range}}(\P_{\widehat{\phi}})$.
\end{remark}
\begin{figure}
  \begin{center}
    \includegraphics[width=0.9\textwidth]{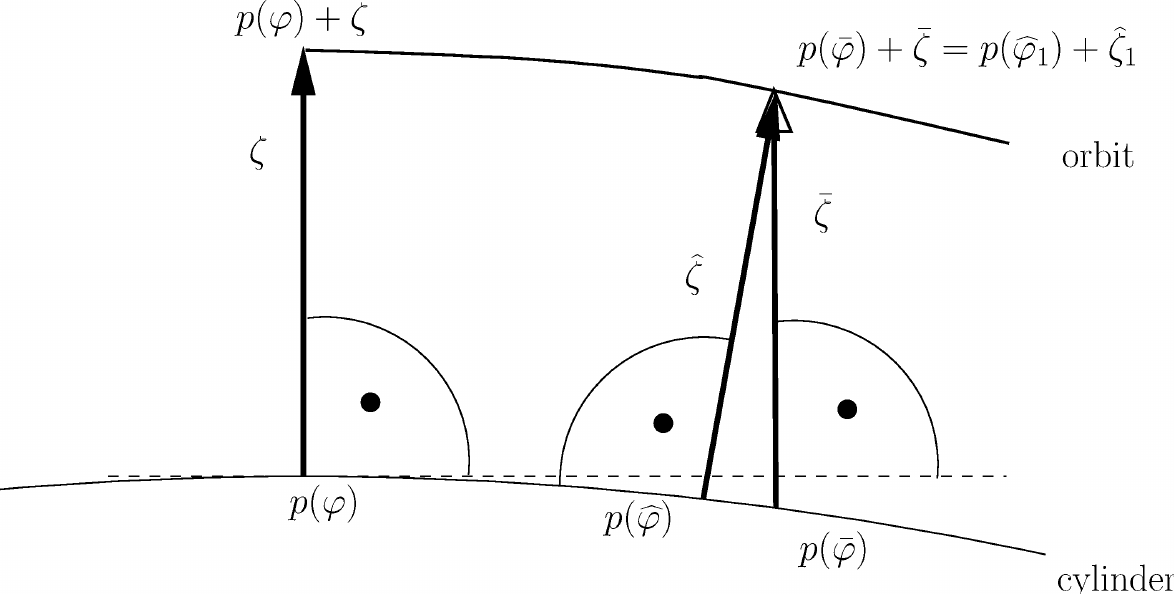}
    \end{center}
\caption{Illustration of the re-orthogonalization process. At time 0,
  the orbit starts at a distance $\|\zeta\|$  from the base point
  $p(\phi)$, which lies on the cylinder (shown as a
line). $\zeta$ is orthogonal to the tangent at the point $p(\phi)$
on the cylinder (this is the 2-dimensional subspace of 0
eigenvalues). At time $T$, the solution has moved to $p(\bar\phi)+\bar\zeta
$, with $\bar\zeta$ still orthogonal to the tangent space at
$p(\phi)$. {The re-orthogonalization consists of finding a new base
point $\widehat\phi$ in such a way that $p(\bar\phi)+\bar\zeta
=p(\widehat\phi) +\widehat\zeta $ and
$\widehat\zeta$ is orthogonal to the tangent space at $p(\widehat
\phi)$. This solution is found by the implicit function
theorem. Note that $\|\widehat\zeta\|$ might be larger than
$\|\zeta \|$, but this is compensated by the contraction induced by semigroup due
to  the dissipation.}}
\label{fig:reorth}       % Give a unique label
\end{figure}
\proof The proof is an application of the implicit function theorem.  Begin by rescaling $\bar{\zeta} \to \mu \bar{\zeta}$,
with $\| \bar{\zeta} \| = 1$.  Then we have $\widehat{\zeta} = e^{\I \bar{\theta}} p(\bar{\phi}) + \mu \bar{\zeta} - e^{\I \widehat{\theta}} p(\widehat{\phi})$.
We wish to choose $(\widehat{\phi},\widehat{\theta})$ so that $\widehat{\zeta}$ is orthogonal to the tangent space at $(\widehat{\phi},\widehat{\theta})$.
Thus, we define 
\begin{equ}
F(\widehat{\phi},\widehat{\theta};\mu) = \left( \begin{array}{c}  \langle n^{(1)}_{ \widehat{\phi},\widehat{\theta} } | \widehat{\zeta} \rangle \\ \langle n^{(2)}_{\widehat{\phi},\widehat{\theta}} |\widehat{\zeta} \rangle \end{array}\right) =  \left( \begin{array}{c}   \langle n^{(1)}_{ \widehat{\phi},\widehat{\theta} } | (e^{\I \bar{\theta}} p(\bar{\phi}) + \mu \bar{\zeta} - e^{\I \widehat{\theta}} p(\widehat{\phi}) )  \rangle \\    \langle n^{(2)}_{ \widehat{\phi},\widehat{\theta} } | (e^{\I \bar{\theta}} p(\bar{\phi}) + \mu \bar{\zeta} - e^{\I \widehat{\theta}} p(\widehat{\phi}) )  \rangle  \end{array}\right)\ ,
\end{equ}
and the theorem follows by finding zeros of this function.

Note that $F(\bar{\phi},\bar{\theta};0) = 0$.   To compute the derivative of $F$ with respect to $(\widehat{\phi},\widehat{\theta})$ we
recall from the previous sections that the derivatives of $e^{\I \theta} p(\phi)$ with respect to $\phi$ and $\theta$ give precisely
the two vectors $v^{(j)}_{\phi,\theta}$ ($ j=1,2$) which span the zero eigenspace.  Thus, by the normalization of the vectors $n^{(j)}_{\widehat{\phi},\widehat{\theta}}$,
we see that
\begin{equ}
D_{\phi,\theta} F |_{\mu =0} = \left( \begin{array}{cc} 1 & 0 \\ 0 & 1\end{array} \right)\ .
\end{equ}
Thus, the implicit function theorem implies that there exists $\mu_0 > 0$ such that for any $ | \mu | < \mu_0$, we have
a solution $F(\widehat{\phi},\widehat{\theta};\mu) = 0$.

\begin{remark} Note that the size of the neighborhood $\mu_0$ on which we have a solution is independent of the base point
$(\phi,\theta)$ --- thus we have good coordinates on a uniform neighborhood of our original family of breathers.
\end{remark}

\begin{remark} Note that the constructive nature of the proof of the implicit function theorem also results on good estimates of the 
size of the solutions of the equation.  In particular, for small $\mu$, there exists a constant $C>0$ such that the
change in the angle and phase can be estimated as:
\begin{equ}\label{eq:reorthog}
| \bar{\phi } - \widehat{\phi} | + | \bar{\theta} - \widehat{\theta} | \le C \mu (| \langle n^{(1)}_{\bar{\phi},\bar{\theta}} | \bar{\zeta} \rangle |
+ | \langle n^{(2)}_{\bar{\phi},\bar{\theta}} | \bar{\zeta} \rangle | )\ .
\end{equ}
\end{remark}

\subsection{The intuitive picture}

Suppose that we  start from a point near our family of breathers, with coordinates $(\phi_0,\theta_0,\zeta_0)$, with
$\zeta_0 \in {\mathrm{Range}}(\P_{\phi_0})$.  We allow the system to evolve for a time $T$ to be specified below.
After this time, we will have reached a point $(\phi_1 = \phi(T), \theta_1 = \theta(T), \zeta_	1 = \zeta(T))$.
In terms of our original variables, this point will be 
\begin{equ}
w_1 = e^{\I (\phi_1 T + \theta_1)} ( p(\phi_1) + z_1)\ ,
\end{equ}
where $z_1 = (\xi_1 + \I \eta_1)$, with $(\xi_1,\eta_1)^\T = \zeta_1$.  The point is that $\zeta_1 $ is no longer orthogonal
to the tangent space to the cylinder of breathers at the point $(\phi_1,\theta_1)$.  This leads to secular growth in $\zeta$,
and eventually, we would loose control of this evolution.  To prevent this, we re-express the point $w_1$
in terms of new variables $(\widehat{\phi},\widehat{\theta},\widehat{\zeta})$, with $\widehat{\zeta}$ orthogonal to the tangent
space at $(\widehat{\phi},\widehat{\theta})$, and restart the evolution of \eqref{eq:zetadot} with
these new initial conditions.    The only complication is that we must keep careful
track of how much we change the various variables in the course of this re-orthogonalization process.
We now explain how this is done.

Without loss of generality assume that we have chosen the ``stopping time'' $T$ so that the phase $ e^{\I (\phi_1 T + \theta_1)} =1$.  (If this is not the case, we can always use the phase invariance of the equation
to rotate the solution so that this does hold.)
Then, after time $T$, the trajectory of our system will have reached the point
\begin{equ}
w_1 = p(\phi_1) + z_1\ .
\end{equ}
By Proposition \ref{prop:IFT} we know that there exists $(\widehat{\phi}, \widehat{\theta}, \widehat{\zeta})$ with
\begin{equ}
w_1 = p(\phi_1) + z_1 = e^{\I \widehat{\theta}} p(\widehat{\phi}) + \widehat{\zeta} \ ,
\end{equ}
and $\widehat{\zeta}$ is normal to the cylinder of breathers at $(\widehat{\phi}, \widehat{\theta})$.  We now restart the evolution of the modulation
equations \eqref{eq:zetadot} and follow the evolution as before.

The last thing we need to control the long-time evolution of the system is to estimate by how much we change $\phi$ and $\zeta$ in the 
course of this re-orthogonalization.  (The change in $\theta$ is inconsequential since it does not affect the magnitude of the solution, and since
the phase-invariance of the equations of motion allows to always rotate the system back to zero phase if needed.)  The change from
$\phi_1$ to $\widehat{\phi}$ is estimated with the aid of the implicit function theorem.

We know that the vectors $\langle n^{(j)}_{\phi,\theta} |$ depend smoothly on $\phi$ and hence
\begin{equa}[eq:phijump]
 | \langle n^{(1)}_{\phi_1} | \zeta_1 \rangle &\le | \langle n^{(1)}_{\phi_1}  | \zeta_1 \rangle -  \langle n^{(1)}_{\phi_0}  | \zeta_1 \rangle | +  
| \langle n^{(1)}_{\phi_0}  | \zeta_1 \rangle |  \\
& \le  | \langle n^{(1)}_{\phi_1}  | \zeta_1 \rangle -  \langle n^{(1)}_{\phi_0}  | \zeta_1 \rangle | \le C \delta(T)  \| \zeta_1 \|
\le  C \delta(T)\gamma  \epsilon^n\ .
\end{equa}
Here, the first inequality just uses the triangle inequality, the second the fact that $\zeta_1$ is orthogonal to $n^{(1)}_{\phi_0} $ by construction,
the third uses Cauchy-Schwarz, plus the smooth dependence of the normal vectors on $\phi$, and the last, the estimate on $\zeta_1$ coming
from Theorem \ref{thm:evolution}.    If we combine this estimate with \eqref{eq:reorthog}, we see that the change in $\phi$ from $\phi_1$ to
$\widehat{\phi}$ produced by the re-orthogonalization is extremely small.

It remains to estimate the corresponding change in $\zeta$ when we replace $\zeta_1$ by $\widehat{\zeta}$.   We have
\begin{equ}
p(\phi_1) + z_1 = e^{\I \widehat{\theta}} p(\widehat{\phi} ) + \widehat{z}\ ,
\end{equ}
where as usual $\widehat{z} = \widehat{\xi} + \I \widehat{\eta} $, with $\widehat{\zeta} = (\widehat{\xi},\widehat{\eta})^\T$.  Again, using the fact that $p(\phi)$ depends
smoothly on $\phi$, plus estimates on the difference in $\phi_1$ and $\widehat{\phi}$ given by \eqref{eq:phijump} and similar estimates for the
$\widehat{\theta}$, we see that 
\begin{equ}\label{eq:zetajump}
\| \zeta_1 - \widehat{\zeta} \| \le C_R \delta(T) \gamma \epsilon^n\ ,
\end{equ}
or 
\begin{equ}\label{eq:reorthog2}
\| \widehat{\zeta} \| \le (1+ C_R \delta(T) \gamma \epsilon^n) \|
\zeta_1\|\ ,
\end{equ}for some finite $R_2$.

\section{Iterating}\label{sec:iter}
{

The estimates of the previous section show that if we take initial conditions for 
\eqref{eq:main} close to the cylinder of breathers for the undamped
equations, and if we express that initial point as
\begin{equ}
w_0 = p(\phi_0) + z_0\ ,
\end{equ}
with $\zeta_0 = (\Re(z_0),\Im(z_0))^\T \in {\mathrm{Range}}(\P_0)$ and $\| \zeta_0 \| \le \gamma \epsilon^n$,
then $\phi$, $\theta$, and $\zeta$ will evolve via \eqref{eq:tphidot2}--\eqref{eq:zetadot2} and after
a time $T = \frac{4 C_n}{\kappa_n \epsilon}$ we will have
\begin{eqnarray}
\phi(T) - \phi_0 &=& -2  \gamma \epsilon^{2n-1} T (1+\OO(\epsilon^{1/2}) )~,  \\
\| \zeta(T) \| & \le & (1-\gamma \epsilon)^{-1} (1+ \frac{3}{2} C_n \gamma) \epsilon^n\ .
\end{eqnarray}
As usual we ignore the evolution of $\theta$ since any $\theta $ dependence of the solution can be
removed using the phase invariance of the problem.  

As discussed in Section \ref{sec:reorthog}, $\zeta(T)$ will not lie in ${\mathrm{Range}}(\P_{\phi(T)})$.
Thus, we now re-orthogonalize.
To see what is involved, consider again \fref{fig:reorth}.

This means we reexpress 
\begin{eqnarray}
w(T) = e^{\I \phi(T) T } p(\phi(T)) + z(T) = e^{\I(\widehat{\phi} T + \widehat{\theta})} p(\widehat{\phi}) + \widehat{z}\ ,
\end{eqnarray}
where as usual, $\widehat{z} = (\widehat{\xi} + \I \widehat{\eta})$, with $(\widehat{\xi},\widehat{\eta}) =\widehat{\zeta} $ and  $\widehat{\zeta} \in {\mathrm{Range}}(\P_{\widehat{\phi}})$.

We now recall the estimates for the change in $\phi$ and $\zeta$ produced by the re-orthogonalization.  First, from \eqref{eq:phijump}, plus the estimate on $\delta(T)$ from Lemma \ref{lem:delta}, we have
\begin{equ}
| \phi(T) - \widehat{\phi}| \le C \delta(T) \epsilon^n~,
\end{equ}
and hence by the triangle inequality we see that
\begin{equ}
| (\phi_0 - \widehat{\phi}) + 2 \gamma \epsilon^{2n-1} T | \le 4 \gamma \epsilon^{2n - 1/2} T\ ,
\end{equ}
\ie  to leading order $\phi_0 - \widehat{\phi} \approx \phi_0 - \phi(T)$.

Likewise, from \eqref{eq:zetajump}, we have
\begin{equa}
\| \widehat{\zeta} \| &\le  \| \zeta(T) \| + \| \zeta(T) - \widehat{\zeta} \|\\& \le e^{-\kappa_n \gamma \epsilon T/2} (1-\gamma \epsilon)^{-1} ( 1 + \frac{3}{2} C_n \gamma) \epsilon^n 
 + 4 \gamma \epsilon^{2n-1} T \\
& \le  \gamma \epsilon^n\ ,
\end{equa}
for $\epsilon $ sufficiently small. If we look at the second line
above, we see how the contraction, and the ``waiting'' for a time $T$
come in: Namely, the first factor contracts,because of the estimates
on the semigroup (an the dissipation), while the next two factors come
from the reprojection and the prefactor from the bound on the semigroup.

Thus, we can begin to evolve our equation of motion starting from the point $w(T)$, but
now expressed as
\begin{equ}
w(T) = e^{\I(T\widehat{\phi}  + \widehat{\theta})} p(\widehat{\phi}) + \widehat{\zeta}\ ,
\end{equ}
where $\widehat\zeta \in {\mathrm{Range}}(\P_{\widehat{\phi}})$, and
$\| \widehat{\zeta} \| \le \gamma \epsilon^n$.  Thus, the new representation for $w(T)$ has the same properties as the representation of $w_0$ that we started with, and
hence we can continue to evolve our trajectory which will remain close
to the cylinder of breathers. 

}
%
%%\section{Appendix}
%\makeatletter
%  \setcounter{section}{0}%
%  \setcounter{subsection}{0}%
%  \renewcommand\thesection{\@Alph\c@section}
%  \makeatother
%\section{Appendix}
%If
%\begin{equ}\Label{eq:lin}
%  \dot y(t) =\alpha (t)y(t)-\beta (t)~,
%  \end{equ}
%  and $y(0)=0$ then we find,
%  with
%  \begin{equa}
%    X(\tau )=\int_0^\tau\d \tau '\, \alpha (\tau ')~.
%  \end{equa}
%  the solution
%    \begin{equa}
%      y(t)&=-\int_0^t\d \tau \,\beta (\tau )e^{X(t)-X(\tau )}
%    \end{equa}
%    If $\alpha $ and $\beta $ vary little over the interval of
%    interest, we have, with
%    \begin{equa}
%     \bar \alpha &=\sup_{\tau \in[0,t]}|\alpha (\tau )-\alpha (0)|~,\qquad
%    \bar\beta =
%    \sup_{\tau \in[0,t]}|\beta (\tau )-\beta (0)|~,
%    \end{equa}
%    a bound of the form
%    \begin{equa}\label{eq:bar}
%      y(t)=\beta (0)t + \OO(\bar \beta )t +\OO\bigl(\alpha (0)\beta
%      (0)+\alpha(0) \bar \beta + \bar \alpha  \beta (0)\bigr) t^2~,
%    \end{equa}
%    valid when $\bar \alpha $ and $\bar \beta$ are
%    small. (We will check the smallness when we use \eref{eq:bar}.)
\begin{acknowledgements}
We thank No\'e Cuneo and Pierre Collet for useful discussions.  The research of CEW was supported
in part by the US NSF under grant number DMS-1813384.
\end{acknowledgements}

    \bibliographystyle{spphys}
\bibliography{refsgene}
\end{document}